\documentclass[twoside,11pt]{article}
\usepackage[utf8]{inputenc}
\DeclareUnicodeCharacter{200B}{}
\DeclareUnicodeCharacter{FB01}{fi}
\DeclareUnicodeCharacter{FB02}{fl}
\DeclareUnicodeCharacter{FB03}{ffi}
\DeclareUnicodeCharacter{FB04}{ffl}
\DeclareUnicodeCharacter{FB00}{ff}
\usepackage{graphicx,amscd,amsfonts,amsmath,amsthm,amssymb,latexsym,amsfonts,color,mathrsfs}
\usepackage{epsfig,float,epstopdf,array,bm,cite}
\usepackage{cases,array}
\usepackage{multirow}
\usepackage{algorithm,algorithmic}
\usepackage{mathtools}
\usepackage{graphicx}
\usepackage{float}
\usepackage{arydshln}
\usepackage{booktabs,comment}
\usepackage[bookmarksnumbered, colorlinks,plainpages]{hyperref}
\newtheorem{theorem}{\bf Theorem}

\newtheorem{proposition}{\bf Proposition}
\newtheorem{example}{\bf Example}

\newcommand{\cred}[1]{{\color{red}  #1}}

\begin{document}
	\title{\bf A parameterized block LU preconditioning for double saddle point problems }
%
\author{\small\bf   
	Hamed Aslani$^{\dag \thanks{Corresponding author. Email: hamedaslani525@gmail.com}}$, 
	Mehdi Najafi Kalyani$^\ddag$,  
	Davod Khojasteh Salkuyeh 
	\thanks{Emails: m.najafi.uk@gmail.com (M.N. Kalyani), khojasteh@guilan.ac.ir (D.K. Salkuyeh)} 
	\\[2mm]
	\textit{{\small $^\dag$Faculty of Mathematical Sciences, University of Guilan, Rasht, Iran}} \\
	\textit{{\small $^\ddag$Department of Computer Science, University of Pisa, Pisa, Italy}}
}

	\date{}
	\maketitle
	\vspace{-0.5cm}
	\noindent\hrulefill\\
	{\bf Abstract.} We present a biparametric preconditioning technique for large, sparse, non-symmetric, and non-singular double saddle point problems. The preconditioner is induced using a stationary iteration method.	It is also based on a two-parameterized block LU factorization of the coefficient matrix. 	To begin with, some convergence results of the iteration method are derived. Moreover, the spectral features exhibited by the preconditioned matrix are given. Finally, the presented preconditioner is conjuncted with the GMRES method. Numerical results demonstrate the satisfactory performance of the preconditioner.
	
	
	\noindent\hrulefill\\
	
	\pagestyle{myheadings}\markboth{H. Aslani, M.N. Kalyani, D.K. Salkuyeh}{A parameterized block LU preconditioning for double saddle point problems}
	\thispagestyle{empty}
	
	\section{Introduction}
	In various scientific and engineering applications, such as  mixed finite element method for approximating elastoplasticity problems \cite{25,26,28}, the use of non-matching meshes in finite element methods to solve stationary model elliptic problems with discontinuous coefficients \cite{1}, a method for solving nonlinear elasticity problems using a mixed formulation of various variables \cite{24}, a technique for analyzing coupled-field phenomena over time involving a hybrid approach that integrates both diffuse element and finite element method (DEM-FEM) \cite{27,28},  a class of linear systems, known as block three-by-three linear systems, is obtained as follows
	\begin{equation}\label{EQ1}
		\mathcal{K} {\bf x} \equiv\begin{pmatrix}
			{A} & {B^{\top}} & {0} \\
			{-B} & {0} & {-C^{\top}} \\
			{0} & {C} & {0}
		\end{pmatrix}\begin{pmatrix}
			{x} \\
			{y} \\
			{z}
		\end{pmatrix}=\begin{pmatrix}
			{f} \\
			{-g} \\
			{h} 
		\end{pmatrix}\equiv \mathbf{b},  
	\end{equation}
	where $A\in \mathbb{R}^{n \times n}$ is a symmetric positive definite (SPD) matrix, matrices  $B\in \mathbb{R}^{m \times n }$ and 
 $C\in \mathbb{R}^{l \times m}$  have full row ranks, $f  \in \mathbb{R}^{n}, g \in \mathbb{R}^{m}$, and $h \in \mathbb{R}^{l}$, while ${\bf{x}}$ represents the unknown vector to be solved. 
	The above assumptions ensure the coefficient matrix $\mathcal{K}$ in linear equation \eqref{EQ1} is nonsingular and the system \eqref{EQ1} has a unique solution. Further details on the conditions that guarantee the nonsingularity of the coefficient matrix \eqref{EQ1} can be found in \cite{Xie, Salkuyeh}. 
		Moreover, similar large-scale structured linear systems also arise in computational electromagnetics, such as electromagnetic scattering from multilayered dielectric objects, where the development of efficient iterative solvers and preconditioning techniques is essential for large-scale simulations \cite{r5c1}.
	
	Linear systems of equations in the form
	\begin{equation}\label{2by2}
		\begin{pmatrix}
			{E} & {F^{\top}}  \\
			{-F} & {G}  
		\end{pmatrix}\begin{pmatrix}
			{x} \\
			{y} 
		\end{pmatrix}=\begin{pmatrix}
			{f} \\
			{g} 
		\end{pmatrix},
	\end{equation}
	are known as the generalized saddle point problems, where $E$ is SPD and $G$ is symmetric positive semi-deﬁnite (SPSD). It should be noted that Eq. \eqref{EQ1} cannot be regarded as a special case of  Eq. \eqref{2by2}. Therefore, existing methods for solving Eq. \eqref{2by2} are not applicable to solve Eq. \eqref{EQ1}, as discussed in \cite{2,3,37,38,39,40,41,42,43,44}. Hence, new approaches are needed for solving double saddle point problems of the form  \eqref{EQ1}.  The iterative solution and preconditioning of three-by-three saddle point problems has garnered considerable attention in recent years \cite{Beik1,Beik2,JunLiJIAM,LiJIAM,BeikAML}.  
	Linear systems of the form \eqref{2by2} also appear in three-dimensional magnetostatic problems discretized by edge finite elements, motivating the development of efficient block preconditioners and Krylov subspace methods \cite{R5c2}.

	As it is well-known,  the coefficient matrix ${\mathcal{K}}$  is typically large and sparse, making iterative methods like Krylov subspace methods appropriate for solving Eq. \eqref{EQ1}. To improve computational efficiency, various preconditioning techniques are used.

	Recently, several preconditioners have been developed and thoroughly analyzed to improve the efficiency of solving \eqref{EQ1}. For instance, Abdolmaleki et al. in \cite{Karimi} proposed the block diagonal preconditioner  
\begin{equation*}
	\mathcal{P}_1=\left(\begin{array}{ccc}
		A & 0& 0 \\
		0 & \alpha I + \beta BB^{\top} & 0\\
		0& 0 & \alpha I +\beta CC^{\top}
	\end{array}\right).
\end{equation*} 
Here and in what follows, 
$I$
denotes the identity matrix of the appropriate dimension. 
Aslani and Salkuyeh in \cite{AslaniJIAM} established  the following preconditioner to the system  \eqref{EQ1}
\begin{equation*}\label{DKS01}
	\mathcal{P}_2=\left(\begin{array}{ccc}
		A & B^{\top} & 0 \\
		0 & \alpha I + \beta BB^{\top} & -C^{\top}\\
		0& 0 & \alpha I +\beta CC^{\top}
	\end{array}\right),
\end{equation*}
where $\alpha$ and $\beta$ are two positive numbers. 
During the application of the preconditioner $\mathcal{P}_1$ or $	\mathcal{P}_2$ within a Krylov subspace method, each iteration requires the solution of a linear system of the form $	\mathcal{P}_1 w=r$ or $	\mathcal{P}_2 w=r$. This process involves solving three auxiliary linear systems whose coefficient matrices are $A$, $\alpha I+ B B^{\top}$, and $\alpha I+ C C^{\top}$, respectively. These systems may be solved either exactly by means of the Cholesky factorization or approximately using an iterative solver such as the conjugate gradient (CG) method. Consequently, the efficiency of the overall algorithms are largely determined by the cost of solving these inner systems. For large-scale applications, this cost may dominate the computational effort of each outer Krylov iteration. However, explicitly constructing $\alpha I+ B B^{\top}$ and $\alpha I+ C C^{\top}$
is prohibitively expensive in terms of both computational time and memory requirements. Furthermore, even when the original matrices $A$, $B$, and $C$ are sparse, $\alpha I+ B B^{\top}$ and $\alpha I+ C C^{\top}$ are generally dense, making their storage and manipulation impractical.  A preconditioner of the form 
	\begin{equation*}
	{\mathcal{P}}_3 = \begin{pmatrix}
		{A} & {B^{\top}} & -\frac{1}{\beta}B^{\top} C^{\top} \\
		-{B} & {\alpha I} & {-C^{\top}} \\
		{0} & {C} & 0
	\end{pmatrix},
\end{equation*}
was presented by Li and Li in \cite{LiJIAM}, where $\alpha$ and $\beta$ are two positive numbers.
 A  block diagonal preconditioner $\mathcal{P}_D$ and its inexact version $\hat{\mathcal{P}}_D$, have been studied in \cite{Huang1}, which are in the form
	\begin{equation*}
		{\mathcal{P}}_D = \begin{pmatrix}
			{A} & {B^{\top}} & {0} \\
			{B} & {S} & {0} \\
			{0} & {0} & CS^{-1}C^{\top}
		\end{pmatrix},\qquad
		\hat{\mathcal{P}}_D = \begin{pmatrix}
			\hat{A} & {B^{\top}} & {0} \\
			{B} &  \hat{S}& {0} \\
			{0} & {0} &C\hat{S}^{-1}C^{\top}
		\end{pmatrix},
	\end{equation*}	
	where $S=BA^{-1} B^{\top}$, $\hat{A}$ and $\hat{S}$ are SPD approximations of $A$ and $S$, respectively.
	While the preconditioner  ${\mathcal{P}}_D$ can effectively cluster the eigenvalues of  ${\mathcal{P}}_D^{-1}\mathcal{K}$,  it is quite time-consuming to construct ${\mathcal{P}}_D$ in practical applications. In a separate line of research, Cao \cite{Cao} introduced the shift-splitting iteration method for solving the system  \eqref{EQ1}. The iteration is given by
	\[
	\frac{1}{2}(\alpha I+\mathcal{K})\mathbf{x}^{k+1}=\frac{1}{2}(\alpha I-\mathcal{K})\mathbf{x}^{k}+\mathbf{b}, \quad k=0,1,\dots,
	\]
 which leads to the shift-splitting preconditioner $\mathcal{P}_{SS}$ as
	\begin{equation}\label{preQ01}
		\mathcal{P}_{SS}=\frac{1}{2}\left(\begin{array}{ccc}
			{\alpha I+{A}} & {B^{\top}} & {0} \\
			-{B} &\alpha I& -C^{\top}\\
			0& C & \alpha I
		\end{array}\right),
	\end{equation}
	and a relaxed version of the shift-splitting 
	preconditioner $\mathcal{P}_{RSS}$ as
	\begin{equation}\label{preQ11}
		\mathcal{P}_{RSS}=\frac{1}{2}\left(\begin{array}{ccc}
			A & {B^{\top}} & {0} \\
			-{B} &\alpha I& -C^{\top}\\
			0& C & \alpha I
		\end{array}\right),
	\end{equation}
	where $\alpha $ is a positive constant and $I$ is the identity matrix with suitable size.  
	Although the shift-splitting method is unconditionally convergent, it exhibits several limitations that may restrict its practical efficiency. For instance, each iteration of the shift-splitting method requires solving shifted linear \textcolor{blue}{systems}. These matrices are better conditioned than the original one. However, solving them still constitutes the dominant computational cost in large‑scale problems. Moreover, the method itself suffers from slow convergence, which further undermines its overall efficiency, especially for large dimensions.  
	Dimensional factorization  preconditioners have been  developed for block-structured saddle point systems arising from the discretization of the incompressible Navier–Stokes equations. In particular, the  dimensional factorization preconditioner, its relaxed version 
	\cite{R6comm12}, 
	 and the stabilized dimensional factorization  preconditioner
	 \cite{R6comm22}
	  provide useful insights into factorization-based preconditioning strategies for saddle point problems. Although the block three-by-three structure of the Navier–Stokes systems shares some similarities with the double saddle point problems considered in this work, the underlying matrix blocks and their algebraic properties are fundamentally different. Nevertheless, these works offer valuable perspectives on the broader landscape of factorization-based preconditioning for  block-structured systems. 
	 Wang et al. \cite{Wang} considered block SPD preconditioners and proposed an exact block preconditioner as follows
	\begin{equation*}
		\mathcal{\tilde{P}}  \equiv \mathcal{L} \mathcal{D}_{\alpha,\beta} {\mathcal{L}}^{\top} = \begin{pmatrix}
			{A} & {B^{\top}} & {0} \\
			{B} & {(\alpha+1)S} & {0} \\
			{0} & {0} & {\beta Q}
		\end{pmatrix},
	\end{equation*}	
	where 
	\begin{equation*}
		\mathcal{L} = \begin{pmatrix}
			{I} & {0} & {0} \\
			{B A^{-1}} & {I} & {0} \\
			{0} & {0} & {I}
		\end{pmatrix}, \qquad \mathcal{D}_{\alpha,\beta} = \begin{pmatrix}
			{A} & {0} & {0} \\
			{0} & {\alpha S} & {0} \\
			{0} & {0} & {\beta Q}
		\end{pmatrix},
	\end{equation*}	
	and	$\alpha,\beta$ are positive constants, $S=BA^{-1} B^{\top}$ and $Q=C S^{-1} C^{\top}$. In particular, the preconditioner is inspired by the symmetric indefinite factorization of the coefficient matrix in  Eq. \eqref{EQ1}:
	\begin{equation*}				
		\mathcal{K}  = \mathcal{L}_{1} \mathcal{L}_{2} \mathcal{D} {\mathcal{L}_{2}}^{\top} {\mathcal{L}_{1}}^{\top},
	\end{equation*}	
	in which
	\begin{equation*}
		\mathcal{L}_{1}=  \begin{pmatrix}
			{I} & {0} & {0} \\
			{B A^{-1}} & {I} & {0} \\
			{0} & {0} & {I}
		\end{pmatrix},\quad \mathcal{L}_{2}=  \begin{pmatrix}
			{I} & {0} & {0} \\
			{0} & {I} & {0} \\
			{0} & {-C S^{-1}} & {I}
		\end{pmatrix}, \quad \mathcal{D}=  \begin{pmatrix}
			{A} & {0} & {0} \\
			{0} & {-S} & {0} \\
			{0} & {0} & {Q}
		\end{pmatrix}.
	\end{equation*}	 
	
It should be noted that forming the Schur complements $S$ and $Q$ 
requires solving expensive matrix inversions, making them very time-consuming.	More recently, Balani et al. \cite{Balani0}, introduced and analyzed an exact block preconditioner for solving  Eq. \eqref{EQ1}, which is
	defined as
	\begin{equation}\label{preQ1}
		\mathcal{Q}=\left(\begin{array}{ccc}
			{{A}} & {B^{\top}} & {0} \\
			{B} &S& -C^{\top}\\
			0& C & \alpha I
		\end{array}\right),
	\end{equation}
	where $\alpha $ is a positive constant and $ S=BA^{-1} B^{\top}$.
	Balani et al. \cite{Balani0} and Wang et al. \cite{Wang}
	demonstrated through experiments that their 
	preconditioners are more effective than $\mathcal{P}_{SS},\mathcal{P}_{RSS}$ and 	${\mathcal{P}}_D$.
	
	
	In this paper, we will describe a new approach of block preconditioning that will be used in the Krylov subspace methods like GMRES and FGMRES to solve a linear system of equations \eqref{EQ1},
	\begin{equation*}
		\mathcal{P}=\begin{pmatrix}
			{I} & {0} & {0} \\
			{\frac{-1}{\beta}B A^{-1}} & {I} & {0} \\
			{0} & {\frac{1}{\alpha}C S^{-1}} & {I}
		\end{pmatrix}\begin{pmatrix}
			{A} & {B^{\top}} & {0} \\
			{0} & {\alpha S} & {-C^{\top}} \\
			{0} & {0} & {\frac{1}{\alpha}Q}
		\end{pmatrix},
	\end{equation*}
	where $\alpha$ and $\beta$ are positive constants,  $S=BA^{-1} B^{\top}$ and {$Q=C S^{-1} C^{\top}$. The paper primarily investigates the spectral properties of the two preconditioned matrices ${\mathcal{P}}^{-1}\mathcal{K}$  and $\hat{\mathcal{P}}^{-1}\mathcal{K}$  in detail. 
	Here, $\hat{\mathcal{P}}$ denotes an inexact version of ${\mathcal{P}}$, and various specific formulations of $\hat{\mathcal{P}}$ will be analyzed.
		 We use an analysis approach inspired by works referenced in \cite{Balani0,Wang} and compare the effectiveness of $\tilde{\mathcal{P}},\mathcal{Q}$, $\mathcal{P}_{SS}$, and $\mathcal{P}$.
		We note that, basically, the cost of applying the preconditioners directly to Eq. \eqref{EQ1} is \textcolor{blue}{very} expensive. Therefore, in practical applications, we prefer to use the inexact version of preconditioners.  
		The proposed preconditioner improves the solution process in three precise senses: it yields a more favorable eigenvalue distribution clustered around one, it reduces the number of Krylov iterations in the exact and inexact settings, and it lowers the practical computational cost by avoiding explicit matrix inversions. The main novelty of the paper lies primarily in the new two-parameter block LU structure; the induced stationary iteration and the spectral analysis provide the theoretical foundation, while the inexact implementation makes the method computationally viable.
		
		%
		
		The remaining content of this paper is organized as follows. 
		Our primary objective is to introduce a new iterative method for solving Eq. \eqref{EQ1} and to extract a preconditioner from this method, which will be discussed  in  Section \ref{Sec2}. In Section \ref{Sec3}, we will analyze the spectrum of the preconditioned matrix. Practical implementation of the preconditioner will be given in 
		 Section \ref{Sec_3}. 
		We will also propose some inexact preconditioners for the matrix $\mathcal{K}$  in Section \ref{Sec_3} and study the eigenvalue distribution of the preconditioned matrix.
		Section  \ref{Sec4} will be dedicated to presenting numerical results. Finally, in Section \ref{Sec5}, we will conclude with some brief remarks.
		
		In this paper, the vector $(x^{\top},y^{\top},z^{\top})^{\top}$ is represented in \textsc{Matlab} notation as $(x;y;z)$, where the superscript $\top$ signifies the transpose operation. Throughout the paper, the notation 
		$x^{*}$
		is employed to signify the conjugate transpose  of the vector 
		$x$. 
		A matrix $A \in \mathbb{R}^{n \times n}$ is considered SPD if it satisfies two conditions: $A^{\top}=A$, and for any nonzero $x$ in $\mathbb{R}^{n}$, $x^{\top} A x>0$. A matrix $A$ is referred to as SPSD if it meets the following requirements: $A^{\top}=A$, and $x^{\top}A x $ must be greater than or equal to zero for all $x$ in $\mathbb{R}^{n}.$ 	

		\section{Suggested iterative method and its convergence}\label{Sec2}
	Motivated by the proposed preconditioner in \cite{Liang}, we give the new iteration method using the splitting $\mathcal{K}=\mathcal{P}-\mathcal{R},$ in which
	\begin{equation}\label{pre}
		\mathcal{P}=\begin{pmatrix}
			{I} & {0} & {0} \\
			{\frac{-1}{\beta}B A^{-1}} & {I} & {0} \\
			{0} & {\frac{1}{\alpha}C S^{-1}} & {I}
		\end{pmatrix}\begin{pmatrix}
			{A} & {B^{\top}} & {0} \\
			{0} & {\alpha S} & {-C^{\top}} \\
			{0} & {0} & {\frac{1}{\alpha}Q}
		\end{pmatrix},
	\end{equation}
	\begin{equation}
		\mathcal{R}=\begin{pmatrix}
			{0} & {0} & {0} \\
			{(1-\frac{1}{\beta}) I } & {I} & {0} \\
			{0} &{0} & {0}
		\end{pmatrix}\begin{pmatrix}
			{B} & {0} & {0} \\
			{0} & {(\alpha-\frac{1}{\beta}) S} &  {0}\\
			{0} & {0} &{0}
		\end{pmatrix},
	\end{equation}
	where $S=BA^{-1} B^{\top}$ and {$Q=C S^{-1} C^{\top}.$} Assuming that $\alpha$ and $\beta$ are positive parameters. The resulting matrix $\mathcal{P}$ is nonsingular. The iterative process for this splitting takes the form 
	\begin{equation}\label{scheme}
		{\bf{x}}^{(k+1)}= \mathcal{G} {\bf{x}}^{(k)}+{\bf{c}}, \qquad k=0,1,2, \dots,
	\end{equation}
		where ${\bf{x}}^{(0)}$ is the initial guess, $\mathcal{G}=\mathcal{P}^{-1} \mathcal{R}$ is the iteration matrix and ${\bf{c}}=\mathcal{P}^{-1} \bf{b}.$ To ensure convergence of the iterative scheme \eqref{scheme}, certain conditions must be met. These conditions will be presented using a relevant theorem. 	
		\begin{theorem}\label{th1}
			Suppose that $A$ is \cred{an} SPD matrix, $B$ and $C$ are matrices that have full row rank. If $\alpha>\frac{1}{2}$, then the iterative method \eqref{scheme} will always converge to the unique solution of \eqref{EQ1}, regardless of the initial guess. 
		\end{theorem}
		\begin{proof}
			Obviously,
			\begin{equation*}
				\mathcal{P}=\begin{pmatrix}
					{A} & {B^{\top}} & {0} \\
					{-\frac{1}{\beta} B} & {(\alpha-\frac{1}{\beta})S} & {-C^{\top}} \\
					{0} & {C} & {0}
				\end{pmatrix}, \quad
				\mathcal{R}=\begin{pmatrix}
					{0} & {0} & {0} \\
					{(1-\frac{1}{\beta}) B} & {(\alpha-\frac{1}{\beta})S} & {0} \\
					{0} & {0} & {0}
				\end{pmatrix}.
			\end{equation*}
			Assuming that $\lambda$ is any eigenvalue of the iteration matrix $\mathcal{G}=\mathcal{P}^{-1}\mathcal{R}$. Let $v$ be the corresponding eigenvector of $\mathcal{G}$ expressed as $(v_1;v_2;v_3).$ We can conclude that $\mathcal{R}v=\lambda \mathcal{P}v$, which is equivalent to 
			\begin{align}
				&\lambda (A v_1 +B^{\top} v_2)=0, \label{eq22}\\
				&\lambda(-\frac{1}{\beta} Bv_1 +(\alpha-\frac{1}{\beta})S v_2 -C^{\top} v_3)=(1-\frac{1}{\beta})B v_1+(\alpha-\frac{1}{\beta})S v_2  \label{eq23}\\
				&\lambda C v_2 =0. \label{eq24}
			\end{align}
			If $\lambda=0,$ there is nothing to prove. So, suppose that $\lambda \neq 0.$ 
			If $v_2 =0,$ based on Eq. \eqref{eq22} and Eq. \eqref{eq23}, both vectors $v_1$ and $v_3$ are also zero since $A$ is SPD and $C$ has full row rank. This conclusion contradicts the fact that $v$ is an eigenvector. Therefore, it must be true that $v_2$ is nonzero.  Now, from Eq.  \eqref{eq22} and as long as positive { definiteness} of $A$, we have $v_1 =-A^{-1} B^{\top} v_2.$ Substituting deduced $v_1$ into Eq. \eqref{eq23}  and multiplying both sides from the left by $v_{2} ^{*}$, gives	
			\begin{equation}\label{eqn1}
			\lambda (\frac{1}{\beta} v_2 ^{*} S v_2 + (\alpha-\frac{1}{\beta}) v_2 ^{*} S v_2-v_{2}^{*}  C^{\top} v_3)= (\frac{1}{\beta} -1) v_2 ^{*} S v_2+ (\alpha-\frac{1}{\beta}) v_2 ^{*} S v_2.
			\end{equation}	 
			By combining Eqs. \eqref{eq24} and \eqref{eqn1}, we have
			$$\lambda (\alpha v_2 ^{*} S v_2) =(\alpha-1) v_2 ^{*} S v_2 .
			$$
			So, $\lambda=1-\frac{1}{\alpha}$, and for the convergence, we need to have $|1-\frac{1}{\alpha}|<1$, which is equivalent to $\alpha>\frac{1}{2}.$
		\end{proof}

		\section{The spectral properties of the preconditioned matrix}\label{Sec3}
		The aim of this section is to obtain the eigenvalues and their corresponding eigenvectors for the preconditioned matrix. The results are summarized in the following theorem.
		\begin{theorem}\label{ttth2}
			Suppose that the { requirements} stated on $A,B,$ and $C$  in Theorem \ref{th1} are satisfied. Then, the eigenpairs of $\mathcal{P}^{-1} \mathcal{K}$, which is denoted by $(\lambda, {\bf{v}}=(v_1;v_2;v_3))$,  falls into the following cases:\\
			{\it{Case I}:} $\lambda=1$, and the corresponding eigenvector is $(v_1;0;v_3)$ in which $v_1$ and $v_3$ arbitrary  but not simultaneously zero if $\beta=1$ and $\alpha \neq1.$ Moreover, when $\alpha=\beta=1$, the corresponding eigenvector presents as $(v_1;v_2;v_3)$, where at least one of $v_i$'s  is nonzero. In addition, if $\alpha \beta=1,$ the corresponding eigenvectors are $(v_1;v_2;v_3)$ during which $v_1 \in\mathrm{null}(B)$ and at least one of $v_i$'s is nonzero. Finally,  $(v_1;\frac{1-\beta}{\alpha \beta -1} S^{-1} B v_1;v_3)$ where $v_1$ and  $v_3$ are not simultaneously zero is the eigenvector in the case of $\alpha \beta \neq 1.$

			\noindent{\it{Case II}:} $\lambda=\frac{1}{\alpha}$, with the corresponding eigenvector $(-A^{-1}B^{\top} v_2;v_2;v_3),$ where $v_{2}\neq0.$		
		\end{theorem}
		\begin{proof}
			Suppose that ${\bf{v}}=(v_1;v_2;v_3)$ represents an eigenvector with the eigenvalue $\lambda$ for the preconditioned matrix.  It means that $\mathcal{P}^{-1} \mathcal{K} {\bf{v}}=\lambda {\bf{v}},$ or $ \mathcal{K} {\bf{v}}=\lambda  \mathcal{P}{\bf{v}},$ equivalently,
			\begin{align}{}
				&(1-\lambda) (A v_1 +B^{\top} v_2)=0, \label{eq26}\\
				&(-1+\frac{\lambda}{\beta}) B v_1+(-1+\lambda) C^{\top} v_3=\lambda (\alpha-\frac{1}{\beta}) Sv_2, \label{eq27}\\
				&(1-\lambda)Cv_2=0 \label{eq28}.
			\end{align}
			Initially, let $\lambda=1.$ In this case, from \eqref{eq27} we get 
			\begin{equation}\label{eq29}
				(-1+\frac{1}{\beta})B v_1=(\alpha-\frac{1}{\beta})Sv_2.
			\end{equation}
			Here, we suppose two occasions for $\beta$, namely $\beta=1$ and $\beta \neq 1.$ As $\beta=1,$ 
			\eqref{eq29} concludes $(\alpha-1) S v_2=0.$ For $\alpha \neq 1,$ we have $v_2=0$ and consequently $(v_1;0;v_3)$ is the corresponding eigenvector, in which $v_1$ and $v_3$ are not zero concurrently. Also, if $\alpha=1,$ $\mathcal{K}=\mathcal{P},$ and each vector in the shape of $(v_1;v_2;v_3)$ as long as at least one of the {components} is nonzero vector can be served as an eigenvector.    
			
			In the following, let $\beta\neq1.$ Eq. \eqref{eq27} implies
			$$(-1+\frac{1}{\beta}) Bv_1=(\alpha-\frac{1}{\beta}) S v_2,$$
			or
			$$(1-\beta) B v_1=(\alpha \beta -1) Sv_2.$$ 
			Now, suppose that $\alpha \beta=1.$ Then, $B v_1 =0$, and the corresponding eigenvector is $(v_1;v_2;v_3),$ in which $v_1 \in $ {null($B$)}. On the other hand, if $\alpha \beta\neq 1,$ we have $v_2= \frac{1-\beta}{\alpha \beta-1} S^{-1} B v_1,$  due to the positive definiteness of $S.$ Overally, $(v_1;\frac{1-\beta}{\alpha \beta-1} S^{-1} B v_1;v_3 )$ is the corresponding eigenvector, with taking the assumption $v_1$ and $v_3$ are not zero at the same time. Finally, let $\lambda \neq 1.$ we can infer that $v_1=-A^{-1}B^{\top}v_{2},$ using Eq.  \eqref{eq26}. 
			In addition, by \eqref{eq27}  we get 
			\begin{equation}\label{eq30} 
				Sv_2 -C^{\top} v_3= \lambda (\alpha Sv_2 -C^{\top} v_3).
			\end{equation}
			Note that $v_2 \neq 0,$ otherwise $v_1=0$ and by Eq. \eqref{eq30}, $v_3 =0,$ and it is impossible. So, \eqref{eq30} is rewritten as follows by multiplying both sides from the left by ${v_2}^{*}$
			$$(1-\lambda \alpha) {v_2}^{*} S v_2=(1-\lambda) {{v_2}^{*} }C^{\top} v_3,$$
			and noting that Eq. \eqref{eq28}, results in $ (1-\lambda \alpha) {v_2}^{*} S v_2.$ Therefore, $\lambda=\frac{1}{\alpha}$ and ${\bf{v}}=(-A^{-1}B^{\top} v_2;v_2;v_3)$, with $v_2 \neq 0.$ 
		\end{proof}

		{
			\begin{proposition}\label{ppp1}
				Suppose that the preconditioner $\mathcal{P}$ is defined as Eq. \eqref{pre}. Then, the preconditioned matrix   $\mathcal{P}^{-1}\mathcal{K}$ satisfies the polynomial equation  $(\lambda-1)^2(\lambda-\frac{1}{\alpha})=0$. 
				Additionally, when $\alpha=1$,  the minimal polynomial of $\mathcal{P}^{-1}\mathcal{K}$ has degree 2, and the GMRES method will  converge to the exact solution within two iterations.
			\end{proposition}
			\begin{proof}
				A straightforward calculation yields 
				\[
				\mathcal{P}^{-1}=
				\begin{pmatrix}
					A^{-1}&-\frac{1}{\alpha}A^{-1}B^{\top}S^{-1}&-A^{-1}B^{\top}S^{-1}C^{\top}Q^{-1}\\
					0&\frac{1}{\alpha}S^{-1}&S^{-1}C^{\top}Q^{-1}\\
					0&0&\alpha Q^{-1}
				\end{pmatrix}
				\begin{pmatrix}
					I&0&0\\
					\frac{1}{\beta}BA^{-1}&I&0\\
					-\frac{1}{\alpha\beta}CS^{-1}BA^{-1}&-\frac{1}{\alpha}CS^{-1}&I
				\end{pmatrix}.
				\]
				Since the eigenvalues of $\mathcal{K}\mathcal{P}^{-1}$ and $\mathcal{P}^{-1}\mathcal{K}$ are equal, a simple calculation shows that
				\[
				\mathcal{K}\mathcal{P}^{-1}=
				\begin{pmatrix}
					I&0&0\\
					\Psi&\frac{1}{\alpha}I-\frac{(1-\alpha)}{\alpha}C^{\top}Q^{-1}CS^{-1}&(1-\alpha)C^{\top}Q^{-1}\\
					0&0&I
				\end{pmatrix}.
				\]
				We omit writing $\Psi$ here because it is irrelevant to the proof.  The characteristic polynomial of above matrix  is defined as
				\begin{align*}
					\det(\lambda I-\mathcal{K}\mathcal{P}^{-1})=(\lambda-1)^n\left|   
					\begin{matrix}
						(\lambda-\frac{1}{\alpha})I+\frac{(1-\alpha)}{\alpha}C^{\top}Q^{-1}CS^{-1}&(\alpha-1)C^{\top}Q^{-1}\\
						0&(\lambda -1)I
					\end{matrix}
					\right|,
				\end{align*}
				which completes the proof.
			\end{proof}
			\section{Inexact block preconditioners for the matrix $\mathcal{K}$}\label{Sec_3}
			In this section, we will explore an inexact version of the preconditioner introduced in Section \ref{Sec2}. We have three different variations of the mentioned preconditioner, depending on whether we use approximations $\hat{A}$, or both $\hat{A}$ and $\hat{S}$, or  all three $\hat{A},\hat{S}$ and $\hat{Q}$. We will describe the eigenvalue distribution of each preconditioned matrix in the form of $\hat{\mathcal{P}}^{-1}\mathcal{K}$ separately. 
			
			Let us consider the preconditioner $\hat{\mathcal{P}}$, which is defined as \begin{equation}\label{24}
				\hat{\mathcal{P}}=\begin{pmatrix}
					{I} & {0} & {0} \\
					{\frac{-1}{\beta} B \hat{A}^{-1}} & {I} & {0} \\
					{0} & {\frac{1}{\alpha}C \hat{S}^{-1}} & {I}
				\end{pmatrix}\begin{pmatrix}
				\hat {A} & {B^{\top}} & {0} \\
					{0} &  {\alpha \hat{ S}} & {-C^{\top}} \\
					{0} & {0} & {\frac{1}{\alpha}\hat{Q}}
				\end{pmatrix},
			\end{equation}
			In the following, we will present three  inexact  versions of the preconditioner $\hat{\mathcal{P}}$:
			\begin{description}
				\item[Case I:] If we work with an approximation $\hat{A}$, the preconditioner given in Eq. \eqref{24} takes the  form
				\begin{equation}\label{25}
					\hat{\mathcal{P}}=\begin{pmatrix}
						\hat{A} & B^{\top}&0\\
						-\frac{1}{\beta}B&(\alpha-\frac{1}{\beta}) B\hat{A}^{-1} B^{\top}& -C^{\top} \\
						{0} &C &0
					\end{pmatrix}.
				\end{equation}
				\item[Case II:] If we work with approximations $\hat{A}$, and $\hat{S}$,  the preconditioner given in Eq. \eqref{24} is represented by
				\begin{equation}\label{26}
					\hat{\mathcal{P}}=\begin{pmatrix}
						\hat{A} & B^{\top}&0\\
						-\frac{1}{\beta}B&\alpha \hat{S}-\frac{1}{\beta}B\hat{A}^{-1} B^{\top}& -C^{\top} \\
						{0} &C &0
					\end{pmatrix}.
				\end{equation}
				\item[Case III:] If we work with approximations  $\hat{A},\hat{S}$ and $\hat{Q}$, the preconditioner given in Eq. \eqref{24} is represented by
				\begin{equation}\label{27}
					\hat{\mathcal{P}}=\begin{pmatrix}
						\hat{A} & B^{\top}&0\\
						-\frac{1}{\beta}B&\alpha \hat{S}-\frac{1}{\beta}B\hat{A}^{-1} B^{\top}& -C^{\top} \\
						{0} &C &\frac{1}{\alpha}\hat{Q}-\frac{1}{\alpha}C \hat{S}^{-1}C^{\top}
					\end{pmatrix}.
				\end{equation}
			\end{description}
			We will now characterize the eigenvalues of the preconditioned matrix $\hat{\mathcal{P}}^{-1}\mathcal{K}$ using the preconditioners \eqref{25} and \eqref{27}. 
			In a similar manner, we can also characterize the eigenvalues of the preconditioned matrix with the  preconditioner \eqref{26}.

			\begin{theorem}
				Suppose that $A$ is a SPD matrix, $B$ and $C$ are matrices with full row rank. Let $\hat{A}$ be the SPD approximation of ${A}$. Define $\hat{\mathcal{P}}$ as in \eqref{25}. 
				If $\lambda$ is an eigenvalue of the preconditioned matrix $\hat{\mathcal{P}}^{-1}\mathcal{K}$ with the eigenvector $(x; y; z)$, and
				$\gamma^A:=\lambda(\hat{A}^{-1}A)\in [\gamma_{\min}^A,\gamma_{\max}^A]$, then
				the real eigenvalues of the preconditioned matrix $\hat{\mathcal{P}}^{-1}\mathcal{K}$ are located in the interval 
				\[
				\left[\mu,\frac{1}{\alpha}( 1+ \frac{1}{\beta}+(\alpha-\frac{1}{\beta})\gamma_{\max}^A)\right],
				\]
				where 
				$ \alpha>\frac{1}{\beta}$ and $\mu>0$.
				Furthermore, the complex eigenvalues are located  in a circle embedded in the complex plane with a radius of 1 and a center at 1.
			\end{theorem}
			\begin{proof}
				Let ${\bf{v}}$ be an eigenvector with the eigenvalue $\lambda$ for the approximated preconditioned matrix,  i.e., $ \mathcal{K} {\bf{v}}=\lambda  \hat{\mathcal{P}}{\bf{v}}$. Define
				\[
				\mathcal{D}=\begin{pmatrix}
					\hat{A}& 0& 0 \\
					0 & \hat{S} &0\\
					0 & 0& I
				\end{pmatrix}, 
				\]
				where $\hat{S}=B\hat{A}^{-1} B^{\top}$.  Define ${\bf{v}}=\mathcal{D}^{-\frac{1}{2}}{\bf{\xi}}$. Therefore,
				the eigenvalue problem can be written as
				\begin{equation}\label{eigv}
					\mathcal{D}^{-\frac{1}{2}}\mathcal{K}\mathcal{D}^{-\frac{1}{2}}\xi=\lambda \mathcal{D}^{-\frac{1}{2}}\hat{\mathcal{P}}\mathcal{D}^{-\frac{1}{2}}\xi.
				\end{equation}
				Then it follows from Eq.  \eqref{eigv} that
				\begin{equation}\label{eq3}
					\begin{pmatrix}
						\tilde{A}& \bar{B} ^{\top} & 0 \\
						-\bar{B} & 0 & -\bar{C}^{\top}\\
						0 & \bar{C}& 0
					\end{pmatrix}
					\begin{pmatrix} x\\y\\z\end{pmatrix}=\lambda
					\begin{pmatrix}
						I & \bar{B}^{\top}&0\\
						-\frac{1}{\beta}\bar{B}&(\alpha-\frac{1}{\beta}) I& -\bar{C}^{\top}\\
						0&\bar{C}&0
					\end{pmatrix}
					\begin{pmatrix} x\\y\\z\end{pmatrix},
				\end{equation}
				where $\tilde{A}={\hat{A}}^{-\frac{1}{2}} A {\hat{A}}^{-\frac{1}{2}}$,
				$\bar{B}={\hat{S}}^{-\frac{1}{2}}B{\hat{A}}^{-\frac{1}{2}}$, and $\bar{C}=C{\hat{S}}^{-\frac{1}{2}}$.  It is known that $\bar{B}\bar{B}^{\top}=I$. From Eq. \eqref{eq3}, we obtain
				\begin{align}
					&\tilde{A}x-\lambda x=(\lambda-1)\bar{B}^{\top}y,\label{ee1}\\
					&( \frac{1}{\beta}\lambda-1)\bar{B}x+(\lambda-1)\bar{C}^{\top}z=\lambda (\alpha-\frac{1}{\beta})y,\label{ee2}\\
					&(1-\lambda)\bar{C}y=0.\label{ee3}
				\end{align}
				If $y = 0$, by Eq. \eqref{ee1}, we have
				$\tilde{A}x=\lambda x$ which means $\lambda\in [\gamma_{\min}^A,\gamma_{\max}^A]$. Hence, we suppose that  $\lambda\neq 1$ and $y \neq  0$. \\
				$(i)$ If $\lambda$ is a real eigenvalue, from Eq.   \eqref{ee1}, we get
				\begin{equation}\label{eqx1}
					x=(\lambda-1)(\tilde{A}-\lambda I)^{-1}\bar{B}^{\top}y.
				\end{equation}
				Substituting Eq. \eqref{eqx1} into Eq. \eqref{ee2} shows that
				\begin{equation*}
					( \frac{1}{\beta}\lambda-1)(\lambda-1)\bar{B}(\tilde{A}-\lambda I)^{-1}\bar{B}^{\top}y+(\lambda-1)\bar{C}^{\top}z-\lambda(\alpha-\frac{1}{\beta})y=0.
				\end{equation*} 
				By multiplying both sides from the left by $\frac{y^\ast}{y^\ast y}$ and using Eq. \eqref{ee3}, we have
				\begin{equation}\label{poly}
					( \frac{1}{\beta}\lambda-1)(\lambda-1)\frac{{(\bar{B}^{\top}y)}^\ast(\tilde{A}-\lambda I)^{-1}(\bar{B}^{\top}y)}{y^\ast y}-\lambda(\alpha-\frac{1}{\beta})=0.
				\end{equation} 
				Since $\tilde{A} - \lambda I$ is either symmetric positive definite or negative definite, there exists a symmetric positive definite matrix $J$ such that
				\[
				\tilde{A} - \lambda I = J^2 \quad \text{or} \quad \tilde{A} - \lambda I = -J^2.
				\]
				Define $g = \bar{B}^{\top} y$ and define $s = J^{-1}g$, we obtain
				\begin{equation*} 
					( \frac{1}{\beta}\lambda-1)(\lambda-1)\frac{g^\ast(\tilde{A}-\lambda I)^{-1}g}{g^\ast g}\cdot\frac{g^\ast g}{y^\ast y}-\lambda(\alpha-\frac{1}{\beta})=0.
				\end{equation*} 
				Hence,
				\begin{equation*}
					( \frac{1}{\beta}\lambda-1)(\lambda-1)\frac{g^\ast g}{g^\ast(\tilde{A}-\lambda I)g}-\lambda(\alpha-\frac{1}{\beta})=0,
				\end{equation*}
				we know that 
				\[
				\frac{g^\ast g}{g^\ast(\tilde{A}-\lambda I)g}=\frac{1}{\gamma^A-\lambda}.
				\]
				Therefore, Eq. \eqref{poly} can be written as
				\begin{equation*}
					q(\lambda)=\frac{\alpha \lambda^2-\left( 1+ \frac{1}{\beta}+(\alpha-\frac{1}{\beta})\gamma^A\right)\lambda+1}{\gamma^A-\lambda},
				\end{equation*}
				if we consider the following condition
				\begin{equation}\label{nume}
					\lambda>\frac{1}{\alpha}( 1+ \frac{1}{\beta}+(\alpha-\frac{1}{\beta})\gamma^A),
				\end{equation}
				then the numerator is positive. Now, we consider a condition where the denominator becomes negative, i.e.,  $ \gamma^A-\lambda<0$. 
				 Subtracting 
				$\gamma^A$
				from both sides of Eq. \eqref{nume} yields
				\[
				\lambda-\gamma^A>\frac{1}{\alpha}( 1+ \frac{1}{\beta}+(\alpha-\frac{1}{\beta})\gamma^{A})-\gamma^{A}=\frac{\beta +1-\gamma^{A}}{\alpha \beta}.
				\]
				Therefore, if $\gamma^{A} <\beta+1$, then
			$\lambda-\gamma^{A}>0,$ so 
			$\gamma^{A} -\lambda<0$, making the denominator of 
			$q(\lambda)$
			 negative for  
				\[\lambda>\frac{1}{\alpha}\left( 1+ \frac{1}{\beta}+(\alpha-\frac{1}{\beta})\gamma_{\max}^A\right).\]
				Now, if $\gamma^A>\frac{1}{2}$, it can be seen  that
				\[
				\lim_{\lambda\to {\gamma^A}^{+}}q(\lambda)=-\infty,\]
				and
				\[  \lim_{\lambda\to {\gamma^A}^{-}}q(\lambda)=+\infty.
				\]
				Therefore, there exists a $\mu\in (\gamma^A-\frac{1}{2},\gamma^A+\frac{1}{2})$ such that $q(\mu)=0$.\\
				
				\noindent$(ii)$ If $\lambda$ is a complex eigenvalue, multiplying Eq. \eqref{ee1} by $x^\ast$ on the left, the
				transposed conjugate of Eq. \eqref{ee2} by $y$ on the right, 
				and   Eq. \eqref{ee3} by $z^\ast$ on the left, we obtain
				\begin{align*}
					&x^\ast\tilde{A}x-\lambda{\|x\|}^2=(\lambda-1)x^\ast\bar{B}^{\top}y,\\
					&-x^\ast\bar{B}^{\top}y-z^\ast\bar{C}y=\bar{\lambda}(-\frac{1}{\beta}x^\ast\bar{B}^{\top}y+(\alpha-\frac{1}{\beta}){\|y\|}^2-z^\ast\bar{C}y),
					\\
					&(1-\lambda)z^\ast\bar{C}y=0.
				\end{align*}
				It follows from $z^\ast\bar{C}y=0$ that
				\begin{align}
					&x^\ast\tilde{A}x-\lambda{\|x\|}^2=(\lambda-1)x^\ast\bar{B}^{\top}y,\label{com1}\\
					&x^\ast\bar{B}^{\top}y=\frac{\bar{\lambda}}{\bar{\lambda}-\beta}(\alpha\beta-1){\|y\|}^2.\label{com2}
				\end{align}
				Substituting Eq. \eqref{com2} into Eq. \eqref{com1} yields that
				\begin{equation}\label{eqrc}
					x^\ast\tilde{A}x-\lambda{\|x\|}^2=\frac{{|\lambda|}^2-\bar{\lambda}}{\bar{\lambda}-\beta}(\alpha\beta-1){\|y\|}^2.
				\end{equation}
				Assume that $\lambda=a+ib$, then the real and imaginary parts of Eq. \eqref{eqrc} are obtained as 
				\begin{align}
					&x^\ast\tilde{A}x-a{\|x\|}^2= \frac{(a^2+b^2-a)(a-\beta)-b^2}{(a-\beta)^2+b^2} (\alpha\beta-1){\|y\|}^2,\label{com3}\\
					&-b\|x\|^2=\frac{(a^2+b^2-a)b+(a-\beta)b}{(a-\beta)^2+b^2} (\alpha\beta-1){\|y\|}^2.\label{com4}
				\end{align}
				Using ${|\lambda-1|}^2-1=a^2+b^2-2a$, and  substituting Eq. \eqref{com4} into Eq. \eqref{com3}  concludes that
				\begin{equation*} 
					x^\ast\tilde{A}x+ \frac{({|\lambda-1|}^2-1)\beta+a^2+b^2}{{(a-\beta)}^2+{b}^2}(\alpha\beta-1)\|y\|^2=0,
				\end{equation*}
				which implies that
				\[
				{|\lambda-1|}^2=1-\frac{{\gamma^{A}}\eta{\|x\|}^2+(a^2+b^2){\|y\|}^2}{\beta{\|y\|}^2},
				\]
				in which 
				$\eta=\frac{ {(a-\beta)}^2+b^2}{\alpha\beta-1}>0$, $	\gamma^{A}=\frac{x^\ast\tilde{A}x}{x^\ast x}$,  and then $\frac{{\gamma^{A}}\eta{\|x\|}^2+(a^2+b^2){\|y\|}^2}{\beta{\|y\|}^2}>0$.
			\end{proof}
		}
	
	Although some intermediate expressions in the proof depend on the eigenvector components, the final bounds have a clear global interpretation. The upper bound for real eigenvalues is independent of the eigenvector; the lower bound can be made global by taking the minimum over all possible 
$\gamma^A\in [\gamma_{\min}^A,\gamma_{\max}^A]$. For complex eigenvalues, the relation  	${|\lambda-1|} <1,$ holds uniformly, meaning that all complex eigenvalues lie strictly inside the unit disk centered at 1. This clustering property is the main theoretical justification for the fast convergence of GMRES when using the inexact preconditioner \eqref{25}. In practice, the parameters 
	$\alpha$
 and 
$\beta$ can be tuned to reduce the disk radius and improve performance, as illustrated in the numerical experiments.
	
		\begin{theorem}\label{th44}
			Suppose that $A$ is  SPD matrix, $B$ and $C$ are matrices with full row rank.
			Let  $\hat{A},\hat{S}$ and $\hat{Q}$  be the SPD approximations of ${A},{S}$ and ${Q}$, respectively. 
			If $\lambda$ is an eigenvalue of the preconditioned matrix $\hat{\mathcal{P}}^{-1}\mathcal{K}$, where $\hat{\mathcal{P}}$ is defined as in Eq. \eqref{27}, then $\lambda=a+ib$ satisfies the following cases:
			\[
			\begin{dcases}
				{|\lambda-1|}^2\leq 1-\frac{1}{\beta}\zeta_{\min}^A&\mathrm{if}\,\,\bar{C}y=0,\\
				|\lambda-1|^2\leq 1-\frac{1}{\beta}\psi_{\min}^A&\mathrm{if}\,\,\bar{C}y\neq 0,
			\end{dcases}
			\]
			where 
			\begin{align*}
				\zeta_{\min}^A&=\frac{{\gamma_{\min}^A}\eta{\|x\|}^2+(a^2+b^2){\|y\|}^2}{{\|y\|}^2},\\
				\psi_{\min}^A&=\frac{1}{\beta}\left(
				\frac{\gamma_{\min}^A{\|x\|}^2((a-\beta)^2+b^2)+(a^2+b^2)\nu_1{\|y\|}^2+(a^2+b^2-2a\beta)\nu_2{\|z\|}^2}{\nu_1{\|y\|}^2+\nu_2{\|z\|}^2}
				\right),
			\end{align*}	
			in which 
			\[
			\begin{array}{llll}
				\rho=\frac{y^\ast \bar{B}\bar{B}^{\top}y}{y^\ast y},&
				\rho_1=\frac{1}{\beta}\frac{y^\ast \bar{B}\bar{B}^{\top}y}{y^\ast y},&
				\rho_2=\frac{z^*\bar{C}\bar{C}^{\top}z}{z^*z},&
				\nu_1=\beta(\alpha-\rho_1),\\
				\nu_2=\frac{\beta}{\alpha}(\rho_2-1),&
				\eta=\frac{ {(a-\beta)}^2+b^2}{\alpha-\frac{1}{\beta}\rho},&
				\bar{B}={\hat{S}}^{\frac{-1}{2}}B{\hat{A}}^{\frac{-1}{2}},&
				\bar{C}={\hat{Q}}^{\frac{-1}{2}}C{\hat{S}}^{\frac{-1}{2}}.
			\end{array}
			\]
		\end{theorem}
		\begin{proof}
			Assume that $\mathcal{D}$ is defined as
			\[ \mathcal{D}=
			\begin{pmatrix}
				\hat{A}& 0& 0 \\
				0 &\hat{S} &0\\
				0 & 0&\hat{Q}
			\end{pmatrix}.
			\]
			Similar to Eq. \eqref{eq3}, the eigenvalue problem 
			$
			\mathcal{D}^{-\frac{1}{2}}\mathcal{K}\mathcal{D}^{-\frac{1}{2}}\xi=\lambda \mathcal{D}^{-\frac{1}{2}}\hat{\mathcal{P}}\mathcal{D}^{-\frac{1}{2}}\xi,
			$
			can be written as
			\begin{equation}\label{eq33}
				\begin{pmatrix}
					\tilde{A}& \bar{B} ^{\top} & 0 \\
					-\bar{B} & 0 & -\bar{C}^{\top}\\
					0 & \bar{C}& 0
				\end{pmatrix}
				\begin{pmatrix} x\\y\\z\end{pmatrix}=\lambda
				\begin{pmatrix}
					I & \bar{B}^{\top}&0\\
					-\frac{1}{\beta}\bar{B}&\alpha I-\frac{1}{\beta}\bar{B}\bar{B}^{\top}& -\bar{C}^{\top}\\
					0&\bar{C}&-\frac{1}{\alpha}(\bar{C}\bar{C}^{\top}-I)
				\end{pmatrix}
				\begin{pmatrix} x\\y\\z\end{pmatrix},
			\end{equation}
			where 
			$\bar{B}={\hat{S}}^{\frac{-1}{2}}B{\hat{A}}^{\frac{-1}{2}}$, and $\bar{C}={\hat{Q}}^{\frac{-1}{2}}C{\hat{S}}^{\frac{-1}{2}}$. 
			It follows from Eq. \eqref{eq33} that
			\begin{align}
				&\tilde{A}x-\lambda x=(\lambda-1)\bar{B}^{\top}y,\label{ee133}\\
				&( \frac{1}{\beta}\lambda-1)\bar{B}x+(\lambda-1)\bar{C}^{\top}z=\lambda (\alpha I-\frac{1}{\beta}\bar{B}\bar{B}^{\top})y,\label{ee233}\\
				&(\lambda-1)\bar{C}y= \frac{1}{\alpha}\lambda(\bar{C}\bar{C}^{\top}-I)z.\label{ee333}
			\end{align}
			If $y = 0$, from Eq. \eqref{ee133}, we deduce that
			$\tilde{A}x=\lambda x$ which means $\lambda\in [\lambda_{\min}(\hat{A}^{-1}A),\lambda_{\max}(\hat{A}^{-1}A)]$. Therefore, we suppose that  $\lambda\neq 1$ and $y \neq  0$. \\
			$(i)$ Let $\bar{C}y=0$. Eq. \eqref{ee333} implies that $z=0$ if $\bar{C}\bar{C}^{\top}-I$ is nonsingular.
			Multiplying Eq. \eqref{ee133} by $x^\ast$ on the left and
			transposed conjugate of Eq. \eqref{ee233} by $y$ on the right, we obtain
			\begin{align}
				&x^\ast\tilde{A}x-\lambda \|x\|^2=(\lambda-1)x^\ast\bar{B}^{\top}y,\label{ee_133}\\
				&( \frac{1}{\beta}\bar{\lambda}-1)x^\ast\bar{B}^{\top}y=\bar{\lambda} (\alpha-\frac{1}{\beta}\rho)\|y\|^2 ,\label{ee_233}
			\end{align}
			where $\rho=\frac{y^\ast \bar{B}\bar{B}^{\top}y}{y^\ast y}$.
			Substituting Eq. \eqref{ee_233} into Eq. \eqref{ee_133} shows that
			\begin{equation}\label{com51}
				x^\ast\tilde{A}x-\lambda \|x\|^2-(\lambda-1)\frac{\bar{\lambda}}{\bar{\lambda}-\beta}\beta (\alpha-\frac{1}{\beta}\rho)\|y\|^2=0.
			\end{equation}
			Assume that $\lambda=a+ib$,  the real and imaginary parts of Eq. \eqref{com51} are obtained as follows
			\begin{align}
				&x^\ast\tilde{A}x-a{\|x\|}^2-  \frac{(a^2+b^2-a)(a-\beta)-b^2}{{(a-\beta)}^2+b^2}  \beta(\alpha-\frac{1}{\beta}\rho){\|y\|}^2=0,\label{com31}\\
				&-b{\|x\|}^2-\frac{(a^2+b^2-a)b+(a-\beta)b}{{(a-\beta)}^2+b^2} \beta(\alpha-\frac{1}{\beta}\rho){\|y\|}^2=0.\label{com41}
			\end{align}
			We know that ${|\lambda-1|}^2-1=a^2+b^2-2a$. 
			Substituting Eq. \eqref{com41} into Eq. \eqref{com31}, concludes that
			\begin{equation} \label{eqro}
				x^\ast\tilde{A}x+ \frac{({|\lambda-1|}^2-1)\beta+a^2+b^2}{{(a-\beta)}^2+b^2}(\alpha-\frac{1}{\beta}\rho)
				{\|y\|}^2=0, 
			\end{equation}
			which implies that ${|\lambda-1|}^2=1-\frac{1}{\beta}\zeta$, where
			\[
			\zeta=\frac{\gamma^A\eta{\|x\|}^2+(a^2+b^2){\|y\|}^2}{{\|y\|}^2},\quad\gamma^A=\frac{x^\ast\tilde{A}x}{x^\ast x}, \quad\eta=\frac{ {(a-\beta)}^2+b^2}{\alpha-\frac{1}{\beta}\rho}.
			\]
			Define
			\[
			\zeta_{\min}^A=\frac{{\gamma_{\min}^A}\eta{\|x\|}^2+(a^2+b^2){\|y\|}^2}{{\|y\|}^2}.
			\]
			If $\frac{1}{\beta}\zeta_{\min}\leq 1$, we have 
			$
			{|\lambda-1|}^2\leq 1-\frac{1}{\beta}\zeta_{\min}^A.
			$
			It follows that
			\[
			1-\sqrt{1-\frac{1}{\beta}\zeta_{\min}^A}\leq \mathscr{R}(\lambda)\leq  1+\sqrt{1-\frac{1}{\beta}\zeta_{\min}^A}.
			\]
			If $\frac{1}{\beta}\zeta_{\min}^A> 1$, then
			there is no $\lambda$ with non-zero imaginary part that satisfies Eq. \eqref{eqro}.

			\noindent$(ii)$ Let $\bar{C}y\neq 0$. Multiplying Eq. \eqref{ee133} by $x^\ast$ on the left,
			transposed conjugate of Eq. \eqref{ee233} by $y$ on the right,  and 
			Eq. \eqref{ee233} by $z^*$ on the left, yields that 
			\begin{align}
				&x^\ast\tilde{A}x-\lambda {\|x\|}^2=(\lambda-1)x^\ast\bar{B}^{\top}y,\label{e1_133}\\
				&( \frac{1}{\beta}\bar{\lambda}-1)x^\ast\bar{B}^{\top}y+ (\bar{\lambda}-1)z^*\bar{C}y=\bar{\lambda} (\alpha-\rho_1){\|y\|}^2 ,\label{e1_233}\\
				&(\lambda-1)z^*\bar{C}y= \frac{1}{\alpha}\lambda(\rho_2-1){\|z\|}^2 \label{e1_333}
			\end{align}
			where $\rho_1=\frac{1}{\beta}\frac{y^\ast \bar{B}\bar{B}^{\top}y}{y^\ast y}$ and $\rho_2=\frac{z^*\bar{C}\bar{C}^{\top}z}{z^*z}$.
			Substituting Eqs. \eqref{e1_233} and \eqref{e1_333}  into Eq. \eqref{e1_133} yields that
			\begin{equation}\label{compl}
				x^\ast\tilde{A}x-\lambda {\|x\|}^2-
				\frac{{|\lambda|}^2-\bar{\lambda}}{\bar{\lambda}-\beta}
				\beta(\alpha-\rho_1){\|y\|}^2+
				\frac{{|\lambda|}^2-{\lambda} }{\bar{\lambda}-\beta}\frac{\beta}{\alpha}(\rho_2-1){\|z\|}^2=0.
			\end{equation}
			Let $\lambda=a+ib,$ and $b\neq 0$. 
			Define $\nu_1=\beta(\alpha-\rho_1)$ and $\nu_2=\frac{\beta}{\alpha}(\rho_2-1)$,
			then we can write the real and imaginary parts of Eq. \eqref{compl} as
			\begin{align}
				&x^\ast\tilde{A}x-a {\|x\|}^2-\frac{(a^2+b^2-a)(a-\beta)-b^2}{{(a-\beta)}^2+b^2}\nu_1{\|y\|}^2+
				\frac{(a^2+b^2-a)(a-\beta)+b^2}{{(a-\beta)}^2+b^2}\nu_2{\|z\|}^2=0,\label{cy1}
				\\
				&-{\|x\|}^2-\frac{a^2+b^2-\beta}{{(a-\beta)}^2+b^2}\nu_1{\|y\|}^2+
				\frac{a^2+b^2-2a+\beta}{(a-\beta)^2+b^2}\nu_2{\|z\|}^2=0. \label{cy2}
			\end{align}
			Substituting Eq. \eqref{cy2} into Eq. \eqref{cy1}, using ${|\lambda-1|}^2-1=a^2+b^2-2a$, yields that
			\begin{equation}\label{compp}
				x^\ast\tilde{A}x+\frac{({|\lambda-1|}^2-1)\beta+a^2+b^2}{{(a-\beta)}^2+b^2}\nu_1{\|y\|}^2
				+\frac{({|\lambda-1|}^2-1)\beta+a^2+b^2+2a\beta}{{(a-\beta)}^2+b^2}\nu_2{\|z\|}^2=0.
			\end{equation}
			A straightforward calculation yields
			$
			{|\lambda-1|}^2=1-\frac{1}{\beta}\psi,
			$
			where
			\[\psi=
			\frac{x^\ast\tilde{A}x((a-\beta)^2+b^2)+(a^2+b^2)\nu_1{\|y\|}^2+(a^2+b^2-2a\beta)\nu_2{\|z\|}^2}{\nu_1{\|y\|}^2+\nu_2{\|z\|}^2}
			>0.
			\]
			Define
			\[
			\psi_{\min}^A=\frac{1}{\beta}\left(
			\frac{\gamma_{\min}^A{\|x\|}^2((a-\beta)^2+b^2)+(a^2+b^2)\nu_1{\|y\|}^2+(a^2+b^2-2a\beta)\nu_2{\|z\|}^2}{\nu_1{\|y\|}^2+\nu_2{\|z\|}^2}
			\right).
			\]
			If $
			\frac{1}{\beta}\psi_{\min}^A\leq 1$, 
			we have 
			$
			{|\lambda-1|}^2\leq 1-\frac{1}{\beta}\psi_{\min}^A.
			$
			It follows that
			\[
			1-\sqrt{1-\frac{1}{\beta}\psi_{\min}^A}\leq \mathscr{R}(\lambda)\leq  1+\sqrt{1-\frac{1}{\beta}\psi_{\min}^A}.
			\]
			If $\frac{1}{\beta}\psi_{\min}^A> 1$, then there exists no $\lambda$ with non-zero imaginary part that satisfies Eq. \eqref{compp}.
		\end{proof}
	}
	
	The bounds in Theorem \ref{th44} are more intricate because they involve three approximations $(\hat{A}, \hat{S}, \hat{Q})$, but they still yield global clustering results. The quantities 
$\zeta_{\min}^A$ and $\psi_{\min}^A$​, though defined using norms of the eigenvector, can be bounded below by constants that depend only on the extreme eigenvalues of $\hat{A}^{-1} A, \rho_1,$ and $\rho_2$​(which are themselves bounded in terms of the approximation qualities). Consequently, in both cases ($\bar{C} y=0$ and $\bar{C}\neq0)$, we obtain a uniform estimate of the form 
	$|\lambda-1|^2\leq 1-\delta$
	for some 
	$\delta>0$, provided the approximations are sufficiently accurate. This guarantees that the eigenvalues are clustered around 1, ensuring robust convergence of the preconditioned FGMRES method.
	
	We note here that we use  $\hat{\mathcal{P}}$  as an inexact preconditioner to compute a vector of the form  ${\bf{v}} = \hat{\mathcal{P}}^{-1}{\bf{w}} $  in each iteration of the Krylov subspace methods, such as GMRES method. To accomplish this, we must solve the following linear system:
	\begin{equation*}
		\begin{pmatrix}
			\hat{A} & B^{\top}&0\\
			-\frac{1}{\beta}B&\alpha \hat{S}-\frac{1}{\beta}B\hat{A}^{-1} B^{\top}& -C^{\top} \\
			{0} &C &\frac{1}{\alpha}\hat{Q}-\frac{1}{\alpha}C \hat{S}^{-1}C^{\top}
		\end{pmatrix}
		\begin{pmatrix} v_1\\v_2\\v_3	\end{pmatrix}=\begin{pmatrix} w_1\\w_2\\w_3	\end{pmatrix}.
	\end{equation*}
	Therefore, we can consider the following algorithm.
	\begin{algorithm}[H]
		\caption{Computation of ${\bf{v}} = \hat{\mathcal{P}}^{-1}{\bf{w}} $.}
		\begin{algorithmic}[1]
			\STATE Solve $\hat{A}u_1=w_1$;
			\STATE Solve $\hat{S}u_2=w_2$;
			\STATE Solve $\hat{Q}v_3=-Cu_2-\frac{1}{\beta}C\hat{S}^{-1}Bu_1+\alpha w_3$;\label{state3}
			\STATE Solve $\hat{S}v_2=\frac{1}{\alpha}(\frac{1}{\beta}Bu_1+w_2+C^{\top}v_3)$;
			\STATE Solve $\hat{A}v_1=w_1-B^{\top}v_2$.
		\end{algorithmic}
		\label{Alg1}
	\end{algorithm}
	
	\section{Numerical experiments}\label{Sec4}
	This section is dedicated to evaluating the effectiveness of the proposed inexact parameterized preconditioner \eqref{27} and comparing its performance with the following preconditioners:
	\begin{equation*}
		\mathcal{\bar{P}}=\left(\begin{array}{ccc}
			\hat{A}& B^{\top}&0\\
			B&\tilde{\alpha} \hat{S}+B\hat{A}^{-1} B^{\top}& 0 \\
			0&0& \tilde{\beta} \hat{Q}
		\end{array}\right),
	\end{equation*}
	where $\tilde{\alpha}=m+n+l$ and $\tilde{\beta}=1/(\tilde{\alpha}+1)$, and 
	\begin{equation}\label{preQ}
		\mathcal{Q}=\left(\begin{array}{ccc}
			{\hat{A}} & {B^{\top}} & {0} \\
			{B} &B\hat{A}^{-1} B^{\top}& -C^{\top}\\
			0& C & \alpha I
		\end{array}\right).
	\end{equation}
	The above  preconditioners are provided in \cite{Wang} and \cite{Balani}, respectively. In \eqref{preQ},
	we only work with an approximate $\hat{A}$, similar to Case I in Section \ref{Sec_3}.
	
	In all test problems,
	$\hat{A}$, $\hat{S}$, and $\hat{Q}$  are the SPD approximations of $A,$ $B {A}^{-1} B^{\top},$ and $C {S}^{-1} C^{\top}$, which are chosen as
	\begin{itemize}
		\item $\hat{A}=\mathrm{diag}(A)$,
		\item $\hat{S}=\mathrm{diag}(B \hat{A}^{-1}B^{\top})$,
		\item $\hat{Q}=\mathrm{tridiag}(C \hat{S}^{-1} C^{\top})$ and ${M}$ is its incomplete Cholesky factor with a dropping tolerance of $10^{-3}$.
	\end{itemize}
	All initial guesses are set to zero. In our implementations, all preconditioners are used as right preconditioners to accelerate the convergence of GMRES and FGMRES methods.  
	For all the preconditioners considered in this paper, when the GMRES method is used, all inner linear systems are solved directly using the \textsc{Matlab} backslash operator. In the inexact implementation, however, the linear systems associated with the coefficient matrices $\hat{A}$ and $\hat{S}$ are still solved directly  by the backslash operator, whereas the inner linear system with coefficient matrix $\hat{Q}$ is solved approximately by the preconditioned conjugate gradient (PCG) method with the preconditioner $MM^{\top}$, using a tolerance of $\mathrm{tol}_{\mathrm{PCG}}=10^{-2}$
	and a maximum of 100 iterations.
	

	In the tables, a hyphen (-) indicates that the solution has not been computed within 500 seconds.
	The experiments were performed on a Laptop with an Intel (R), Core (TM) i5-7200U, CPU @ 2.50 GHz and 8 GB of memory, using \textsc{Matlab-R2021}a. The methods being tested were compared based on the number of iteration steps (referred to as ``IT") and the elapsed CPU time in seconds (referred to as ``CPU"). The iteration process was stopped as soon as the current relative residual 2-norm was less than $10^{-7}$, equivalently,
	\[\mathrm{RES}=\frac{{\left\|{\mathbf{b}}-\mathcal{K} {\bf{x}}^{(k)}\right\|}}{{\left\|{\mathbf{b}}\right\|}} \leq 10^{-7},\] 
	or if the maximum number of iterations exceeded 1000. Additionally, the	accuracy of the methods was compared using
	$
	\mathrm{ERR}={{\|{\bf{x} }^{(k)}-{\bf{x} }^\ast  \|}}/{{\|{\bf{x} }^\ast\|} },
	$
	where ${\bf{x} }^{(k)}$ and ${\bf{x} }^{\ast}$ represent the current iteration and the exact solution of \eqref{EQ1}, respectively.
	 Additionally, the right-hand side vector
	${\bf{b} }$ is set as ${\bf{b}}= \mathcal{K}{\bf{e}}$, where ${\bf{e} }\in \mathbb{R}^{n+m+l}$ is a vector of all ones. 
	
		In all tests, for preconditioner \eqref{27}, $\beta$ is assumed to be  $1$, while $\alpha$ is treated as a variable parameter. Based on the theoretical results presented in Theorem \ref{ttth2} and Proposition \ref{ppp1}, the choice $\beta=1$ is theoretically well motivated, since it leads to particularly favourable spectral properties of the preconditioned matrix and may result in rapid convergence of Krylov subspace methods. With $\beta$ fixed at this value, $\alpha$ becomes the primary tuning parameter, and its influence on the convergence behaviour is investigated through the numerical experiments. We emphasize, however, that these parameter choices are not optimal, and determining optimal or near-optimal values of $\alpha$ and $\beta$ remains an important topic for further investigation.
	
	\begin{example}[\cite{Huang1,Xie}]\label{EX1}\rm
		Consider the saddle point problem \eqref{EQ1} for which
		\[
		A=\begin{pmatrix}
			I \otimes T +T\otimes I&0\\
			0&I \otimes T +T\otimes I
		\end{pmatrix}\in \mathbb{R}^{2p^2\times 2p^2},
		\]
		$B=
		(I \otimes F \,\,\, F\otimes I)
		\in \mathbb{R}^{p^2\times 2p^2},
		$ and $C=E\otimes F \in  \mathbb{R}^{p^2\times p^2}$, where
		\[
		T=\frac{1}{h^2}\cdot \mathrm{tridiag}(-1,2,-1)\in \mathbb{R}^{p\times p},\,\quad F=\frac{1}{h}\cdot \mathrm{tridiag}(0,1,-1)\in \mathbb{R}^{p\times p}
		\]
		and $E=\mathrm{diag}(1,p+1,2p+1,\dots,p^2-p+1)$, where $\otimes$ denotes the Kronecker product and
		$h=\frac{1}{p+1}$ represents the discretization mesh size.
	\end{example}	
	In Tables \ref{tabl1} and \ref{tabl2}, we present numerical results obtained from the right preconditioned GMRES and FGMRES methods using different preconditioners.  
	The numerical results clearly demonstrate that the GMRES (FGMRES) method with the preconditioner  $\hat{\mathcal{P}}$  achieves higher accuracy compared to the preconditioners $\bar{\mathcal{P}}$ and $\mathcal{Q}$.  
	For the case $\mathcal{Q}(\alpha=0.5)$ with problem size 262411,  GMRES method stagnated, i.e., two consecutive iterates became identical before the prescribed tolerance was achieved. Consequently, the final relative residual remained at $1.37\times10^{-7}$, which is slightly greater than  the stopping criterion $10^{-7}$.

	Furthermore, Figure  \ref{fi1}  displays the eigenvalue distribution of matrices, $\mathcal{K}$ and $\mathcal{\hat{P}}^{-1}\mathcal{K}$ for $p=32$.
	As can be seen from the figure, the eigenvalues of  $\mathcal{\hat{P}}^{-1}\mathcal{K}$ are clustered around  1 exhibit a higher degree of clustering compared to those of $\mathcal{K}$. This clustering suggests that the corresponding GMRES (FGMRES) method should demonstrate a faster convergence rate.

	In addition to the preconditioners reported in the numerical tables, we experimented with the shift-splitting preconditioner proposed in \cite{Cao}. Although satisfactory performance was observed for some small-scale test problems, the method became significantly less effective as the problem dimension increased, either failing to converge within the prescribed iteration limit or requiring considerably longer computational times. Since these results were not competitive with those presented in the tables and would not provide additional insight into the large-scale behaviour of the methods, they have been omitted for brevity.

	As clearly observed in Figure \ref{EX1FIG}, increasing the parameter $\alpha$ across the broad range of 0.5 to 10 (with a step size of 0.5) results in an almost entirely flat and stable behaviour for both the iteration counts (IT) and the CPU time. The number of iterations required for convergence remains practically constant throughout this entire interval, demonstrating that the algorithm achieves fast convergence almost independently of the specific $\alpha$ value chosen. Likewise, the CPU time exhibits a uniform trend without any significant increase, implying that the computational cost is consistent and does not degrade as $\alpha$ varies. This notably stable response serves as strong evidence of the high robustness of the proposed preconditioner against variations in $\alpha$, confirming that its overall performance is not critically sensitive to the exact tuning of this parameter.
	
	\begin{table}[!htp]
		\centering
		\caption{Numerical results of the preconditioned GMRES method  for Example \ref{EX1}.}
		\begin{tabular}{llllll}
		\toprule
			&Size&4096&16384&65536&262144\\
			&$p$&32&64&128&256\\
			\midrule $\hat{\mathcal{P}}(\alpha=1.5)$
			&IT&53&100&189&353\\
			&CPU&0.11&0.47&4.33&121.46\\
			&RES&$7.45\cdot10^{-08}$&$9.74\cdot10^{-08}$&$9.65\cdot10^{-08}$&$9.66\cdot10^{-08}$\\
			&ERR&$1.47\cdot10^{-06}$&$1.22\cdot10^{-06}$&$1.14\cdot10^{-05}$&$5.55\cdot10^{-05}$\\
			$\hat{\mathcal{P}}(\alpha=5)$
			&IT&53&100&189&353\\
			&CPU&0.08&0.52&4.30&124.06\\
			&RES&$7.52\cdot10^{-08}$&$9.74\cdot10^{-08}$&$9.67\cdot10^{-08}$&$9.46\cdot10^{-08}$\\
			&ERR&$1.14\cdot10^{-06}$&$1.22\cdot10^{-06}$&$9.18\cdot10^{-06}$&$5.53\cdot10^{-05}$\\
			$\hat{\mathcal{P}}(\alpha=10)$
			&IT&53&100&189&353\\
			&CPU&0.08&0.51&4.40&124.70\\
			&RES&$7.53\cdot10^{-08}$&$9.74\cdot10^{-08}$&$9.67\cdot10^{-08}$&$9.45\cdot10^{-08}$\\
			&ERR&$6.75\cdot10^{-07}$&$1.22\cdot10^{-06}$&$7.92\cdot10^{-06}$&$5.53\cdot10^{-05}$\\
			\midrule ${\mathcal{Q}}(\alpha=0.5)$
			&IT&56&101&190&353\\
			&CPU&0.17&1.09&10.23&195.23\\
			&RES&$9.04\cdot10^{-08}$&$9.80\cdot10^{-08}$&$9.62\cdot10^{-08}$&$1.37\cdot10^{-07}$\\
			&ERR&$3.60\cdot10^{-06}$&$2.82\cdot10^{-06}$&$1.31\cdot10^{-05}$&$6.09\cdot10^{-05}$\\
			${\mathcal{Q}}(\alpha=10)$
			&IT&81&104&193&357\\
			&CPU&0.26&1.10&10.47&196.55\\
			&RES&$9.69\cdot10^{-08}$&$9.58\cdot10^{-08}$&$8.81\cdot10^{-08}$&$9.32\cdot10^{-08}$\\
			&ERR&$6.31\cdot10^{-06}$&$1.42\cdot10^{-05}$&$1.61\cdot10^{-06}$&$5.51\cdot10^{-05}$\\ \hline
			$\bar{\mathcal{P}}$
			&IT&85&131&235&429\\
			&CPU&0.15&0.68&6.18&186.11\\
			&RES&$9.04\cdot10^{-08}$&$9.59\cdot10^{-08}$&$9.65\cdot10^{-08}$&$9.75\cdot10^{-08}$\\
			&ERR&$1.94\cdot10^{-06}$&$5.22\cdot10^{-06}$&$8.01\cdot10^{-06}$&$5.77\cdot10^{-05}$\\
			
			\bottomrule
		\end{tabular}
		\label{tabl1}
	\end{table}	

	\begin{figure}[H]
		\centering
		\includegraphics[width=.8\linewidth]{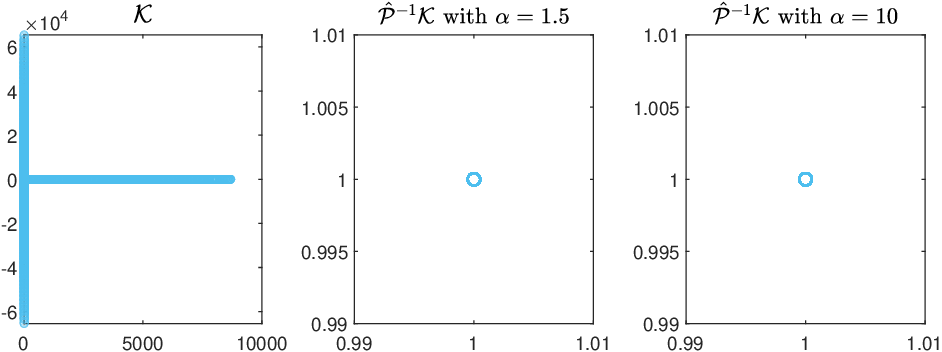}
		\caption{Eigenvalue distribution with $p=32$ for Example \ref{EX1}.}
		\label{fi1}
	\end{figure}
\setlength{\heavyrulewidth}{0.15em} 
\setlength{\lightrulewidth}{0.08em} 
\setlength{\cmidrulewidth}{0.08em}
	\begin{table}[!htp]
		\centering
		\caption{Numerical results of the FGMRES method  for Example \ref{EX1}.		\label{tabl2}}
		\resizebox{\textwidth}{!}{
			\begin{tabular}{llllllll}
				\toprule
				&Size&16384&65536&262144&1048576&4194304&16777216\\
				&$p$&64&128&256&512&1024&2048\\
				\midrule
				$\hat{\mathcal{P}}(\alpha=1.5)$
				&IT&2&2&2&2&2&2\\
				&CPU&0.06&0.23&1.14&5.45&23.90&201.31\\
				&RES&$4.91\cdot10^{-14}$&$1.36\cdot10^{-13}$&$3.81\cdot10^{-13}$&$1.21\cdot10^{-12}$&$3.30\cdot10^{-12}$&$1.04\cdot10^{-11}$\\
				&ERR&$4.13\cdot10^{-13}$&$9.29\cdot10^{-13}$&$1.80\cdot10^{-11}$&$1.61\cdot10^{-08}$&$6.71\cdot10^{-09}$&$1.58\cdot10^{-09}$\\
				$\hat{\mathcal{P}}(\alpha=5)$
				&IT&2&2&2&2&2&2\\
				&CPU&0.08&0.26&1.17&5.38&24.14&202.01\\
				&RES&$4.77\cdot10^{-14}$&$1.49\cdot10^{-13}$&$4.31\cdot10^{-13}$&$1.22\cdot10^{-12}$&$3.16\cdot10^{-12}$&$8.89\cdot10^{-12}$\\
				&ERR&$3.20\cdot10^{-13}$&$7.09\cdot10^{-13}$&$1.80\cdot10^{-11}$&$2.43\cdot10^{-09}$&$6.66\cdot10^{-08}$&$1.38\cdot10^{-08}$\\
				$\hat{\mathcal{P}}(\alpha=10)$
				&IT& 2&2&2&2&2&2\\
				&CPU&0.04&0.24&1.13&5.42&24.17&210.86\\
				&RES&$4.89\cdot10^{-14}$&$1.42\cdot10^{-13}$&$3.82\cdot10^{-13}$&$1.23\cdot10^{-12}$&$3.44\cdot10^{-12}$&$8.64\cdot10^{-12}$\\
				&ERR&$3.86\cdot10^{-13}$&$7.09\cdot10^{-13}$&$1.76\cdot10^{-11}$&$6.35\cdot10^{-10}$&$1.70\cdot10^{-07}$&$6.01\cdot10^{-08}$\\
				\midrule	${\mathcal{Q}}(\alpha=0.5)$
				&IT&41&42&12&&&\\		
				&CPU&1.06&10.36&200.05&-&-&-\\		
				&RES&$9.60\cdot10^{-08}$&$9.81\cdot10^{-08}$&$9.99\cdot10^{-08}$&-&-&-\\		
				&ERR&$1.76\cdot10^{-05}$&$2.95\cdot10^{-04}$&$2.11\cdot10^{-03}$&-&-&-\\		
				${\mathcal{Q}}(\alpha=10)$
				&IT&37&12&11&-&-&-\\		
				&CPU&1.05&10.54&204.46 &-&-&-\\		
				&RES&$9.25\cdot10^{-08}$&$9.82\cdot10^{-08}$&$9.98\cdot10^{-08}$&-&-&-\\		
				&ERR&$2.32\cdot10^{-05}$&$2.58\cdot10^{-04}$&$2.12\cdot10^{-03}$&-&-&-\\	\hline	
					$\bar{\mathcal{P}}$
				&IT&33&21&34&-&-&-\\		
				&CPU&1.31&16.67&287.08&-&-&-\\		
				&RES&$9.53\cdot10^{-08}$&$9.98\cdot10^{-08}$&$9.98\cdot10^{-08}$&-&-&-\\		
				&ERR&$5.74\cdot10^{-05}$&$3.75\cdot10^{-04}$&$2.15\cdot10^{-03}$&-&-&-\\		
				\bottomrule
		\end{tabular}}
	\end{table}	
		\begin{figure}
		\begin{center}
			\includegraphics[height=5.5cm,width=6.5cm]{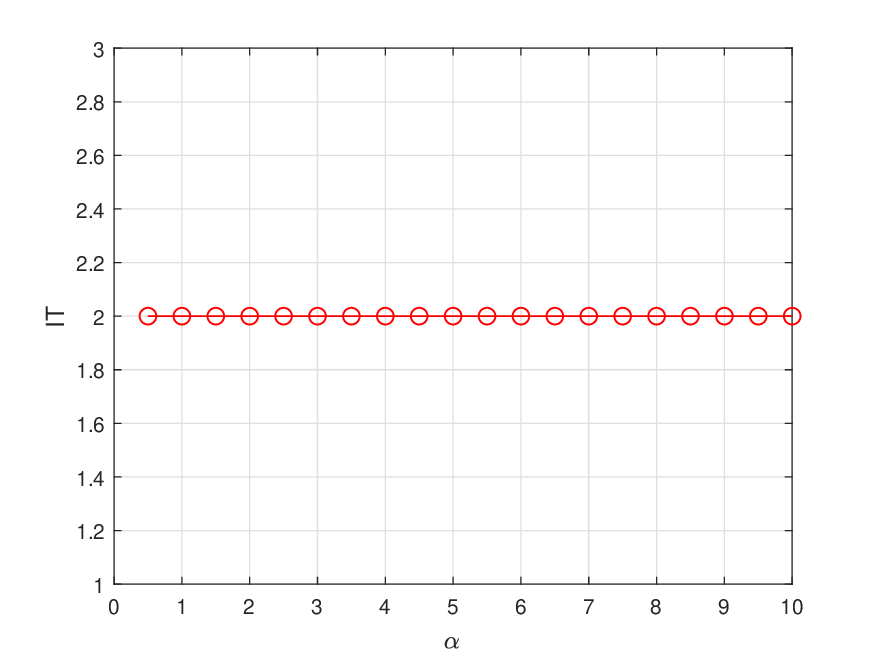}\includegraphics[height=5.5cm,width=6.5cm]{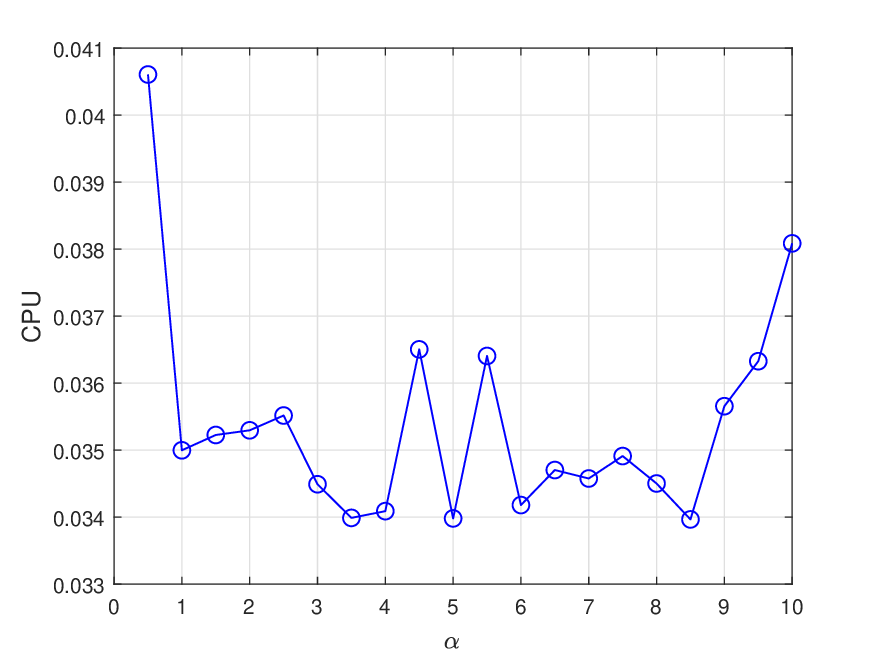}\\
		\end{center}
		\caption{{\small Variation of  iteration counts (IT) and  CPU time with respect to $\alpha$ for Example  \ref{EX1}}}. \label{EX1FIG} 	
	\end{figure}
	\begin{example}[\cite{Huang,Wang}]\label{EX2}\rm
		The problem at hand involves a $3\times 3$ block matrix, where the blocks $A$ and $B$ are derived from solving the Stokes equations given by
		\begin{equation*}
			\begin{dcases}
				-\Delta \mathbf{u}+\nabla p=0, & \text { in } \,\Omega, \\
				\nabla \cdot \mathbf{u}=0, & \text { in } \,\Omega,
			\end{dcases}
		\end{equation*} 
		for a two-dimensional lid-driven cavity within a square domain. The boundary conditions dictate zero flow on the sides and bottom of the domain, and a non-zero horizontal velocity on the lid. To generate the matrices $A$ and $B$, we utilize the IFISS software developed by Elman et al. \cite{IFISS} and employ the finite element method $Q2-P1$ on five grids with varying parameters. Since the block $B$ generated by the IFISS package is not full row rank, we drop the first two rows to obtain a full row rank matrix. We set $h= \frac{1}{16}, \ldots, \frac{1}{128}$. Additionally, we introduce $C$ to make the linear system ill-conditioned but not too sparse. This matrix has a specific form involving a diagonal matrix and a randomly generated matrix, where $l = m - 2$ and $\mathrm{randn}(l,m-l)$ is an $l-by-(m-l)$ matrix of normally distributed random numbers.  In this example, we replace the coefficient matrix $\mathcal{K}$ by
		the matrix $\mathcal{D}^{-\frac{1}{2}}\mathcal{K}\mathcal{D}^{-\frac{1}{2}}$, where
		$\mathcal{D}=\mathrm{diag}({\|\mathcal{K}_1\|},\dots,{\|\mathcal{K}_{n+m+l}\|})$
		and $\mathcal{K}_i$ is the $i$th column of the matrix $\mathcal{K}$.

		Tables \ref{tabl3} and \ref{tabl4} displayed numerical outcomes for Example \ref{EX2} using different $h$, with $Q2-P1$ FEM on  uniform grids. The results in these tables indicated that the suggested preconditioner performs much better than $\bar{\mathcal{P}}$ and $\mathcal{Q}$, as it requires fewer iterations and less CPU time.
		It should be noted that even with the change of {$\alpha$}, the superiority of the proposed method is still maintained. In other words, the suggested preconditioner is better than $\bar{\mathcal{P}}$.
		Furthermore, Figure  \ref{fi2} displays the eigenvalue distribution of matrices, $\mathcal{K}$ and $\mathcal{\hat{P}}^{-1}\mathcal{K}$ for $h=\frac{1}{16}$,
		where the clustering of eigenvalues resulting from the preconditioner 
		$\mathcal{\hat{P}}$
		 is readily observable.
		
		According to Figure \ref{EX2FIG}, as $\alpha$ increases from $0$ to $3$, the iteration count first drops from $14$ to $10$ and then rises to $13$; for $\alpha\geq 3$, the number of iterations remains nearly constant at $13$, indicating fast convergence that is largely independent of $\alpha$ for larger values.	The CPU time similarly decreases from $0.89$ to $0.62$, then exhibits minor fluctuations and stabilizes around $0.81$ for $\alpha\geq 3$, reflecting a consistent and uniform computational cost. This flat behavior for both metrics over a wide range of $\alpha$ (especially for $\alpha\geq 3$) clearly confirms the robustness of the algorithm with respect to the tuning parameter, demonstrating that the method's performance is not sensitive to the exact choice of $\alpha$.
		
		\begin{table}[!htp]
			\centering
			\caption{Numerical results of the preconditioned GMRES method for Example \ref{EX2}.}
			\setlength{\heavyrulewidth}{0.15em} 
			\setlength{\lightrulewidth}{0.08em} 
			\setlength{\cmidrulewidth}{0.08em}
			\begin{tabular}{llllll}
				\toprule
				&$h$ &$\frac{1}{16}$&$\frac{1}{32}$&$\frac{1}{64}$&$\frac{1}{128}$\\
				
				\midrule
				$\hat{\mathcal{P}}(\alpha=1.5)$
				&IT&44&64&109&174\\
				&CPU&0.04&0.10&0.61&9.13\\
				&RES&$6.83\cdot 10^{-08}$&$8.98\cdot 10^{-08}$&$9.67\cdot 10^{-08}$&$8.87\cdot 10^{-08}$\\
				&ERR&$8.28\cdot 10^{-07}$&$8.27\cdot 10^{-06}$&$1.84\cdot 10^{-05}$&$7.44\cdot 10^{-06}$\\
				$\hat{\mathcal{P}}(\alpha=5)$
				&IT&44&59&103&174\\
				&CPU&0.02&0.08&0.58&9.06\\
				&RES&$6.94\cdot 10^{-08}$&$9.84\cdot 10^{-08}$&$9.20\cdot 10^{-08}$&$9.15\cdot 10^{-08}$\\
				&ERR&$3.18\cdot 10^{-06}$&$1.23\cdot 10^{-05}$&$2.36\cdot 10^{-05}$&$5.88\cdot 10^{-06}$\\
				\midrule 
				${\mathcal{Q}}(\alpha=0.5)$
				&IT&40&58&93&171\\
				&CPU&0.10&0.49&3.25&41.49\\
				&RES&$9.33\cdot 10^{-08}$&$9.20\cdot 10^{-08}$&$9.42\cdot 10^{-08}$&$8.48\cdot 10^{-08}$\\
				&ERR&$1.36\cdot 10^{-06}$&$1.15\cdot 10^{-05}$&$2.44\cdot 10^{-05}$&$4.18\cdot 10^{-06}$\\
				${\mathcal{Q}}(\alpha=1.5)$
				&IT&41&59&92&171\\
				&CPU&0.06&0.42&3.10&41.61\\
				&RES&$7.09\cdot 10^{-08}$&$9.12\cdot 10^{-08}$&$9.91\cdot 10^{-08}$&$8.58\cdot 10^{-08}$\\
				&ERR&$8.33\cdot 10^{-07}$&$1.05\cdot 10^{-05}$&$2.44\cdot 10^{-05}$&$4.24\cdot 10^{-06}$\\
				${\mathcal{Q}}(\alpha=5)$
				&IT&42&60&92&171\\
				&CPU&0.08&0.39&3.10&40.28\\
				&RES&$8.48\cdot 10^{-08}$&$9.44\cdot 10^{-08}$&$9.57\cdot 10^{-08}$&$8.84\cdot 10^{-08}$\\
				&ERR&$1.07\cdot 10^{-06}$&$9.01\cdot 10^{-06}$&$2.31\cdot 10^{-05}$&$4.56\cdot 10^{-06}$\\
				\hline
				 $\bar{\mathcal{P}}$
				&IT&101&145&204&253\\
				&CPU&0.07&0.21&1.40&12.03\\
				&RES&$8.60\cdot 10^{-08}$&$9.92\cdot 10^{-08}$&$9.98\cdot 10^{-08}$&$8.81\cdot 10^{-08}$\\
				&ERR&$3.77\cdot 10^{-06}$&$5.22\cdot 10^{-04}$&$1.31\cdot 10^{-03}$&$3.56\cdot 10^{-03}$\\
				\bottomrule
			\end{tabular}
			\label{tabl3}
		\end{table}	
		\begin{table}[!htp]
			\centering
			\caption{Numerical results of the FGMRES method for Example \ref{EX2}.}
			\begin{tabular}{llllll}
				\toprule
				&$h$ &$\frac{1}{16}$&$\frac{1}{32}$&$\frac{1}{64}$&$\frac{1}{128}$\\
				\hline
				$\hat{\mathcal{P}}(\alpha=1.5)$
				&IT&10&9&11&11\\
				&CPU&0.04&0.17&0.93&4.11\\
				&RES&$8.03\cdot 10^{-08}$&$9.67\cdot 10^{-08}$&$3.77\cdot 10^{-08}$&$5.47\cdot 10^{-09}$\\
				&ERR&$8.04\cdot 10^{-06}$&$8.26\cdot 10^{-05}$&$8.08\cdot 10^{-05}$&$2.11\cdot 10^{-05}$\\
				$\hat{\mathcal{P}}(\alpha=5)$
				&IT&12&11&13&13\\
				&CPU&0.05&0.20&1.08&4.83\\
				&RES&$5.76\cdot 10^{-08}$&$4.24\cdot 10^{-08}$&$9.27\cdot 10^{-08}$&$9.51\cdot 10^{-08}$\\
				&ERR&$2.43\cdot 10^{-05}$&$1.24\cdot 10^{-04}$&$2.74\cdot 10^{-03}$&$6.18\cdot 10^{-03}$\\
				\hline
				 ${\mathcal{Q}}(\alpha=0.5)$
				&IT&40&9&36&30\\
				&CPU&0.09&0.21&2.69&19.53\\
				&RES&$9.34\cdot 10^{-08}$&$9.32\cdot 10^{-08}$&$9.57\cdot 10^{-08}$&$9.96\cdot 10^{-08}$\\
				&ERR&$1.40\cdot 10^{-06}$&$1.12\cdot 10^{-05}$&$2.61\cdot 10^{-05}$&$2.12\cdot 10^{-04}$\\
				${\mathcal{Q}}(\alpha=1.5)$
				&IT&41&10&36&32\\
				&CPU&0.04&0.16&2.13&22.03\\
				&RES&$6.97\cdot 10^{-08}$&$9.41\cdot 10^{-08}$&$9.95\cdot 10^{-08}$&$9.69\cdot 10^{-08}$\\
				&ERR&$8.09\cdot 10^{-07}$&$1.05\cdot 10^{-05}$&$2.81\cdot 10^{-05}$&$1.98\cdot 10^{-04}$\\
				${\mathcal{Q}}(\alpha=5)$
				&IT&42&12&37&28\\
				&CPU&0.03&0.16&1.88&23.58\\
				&RES&$8.44\cdot 10^{-08}$&$8.40\cdot 10^{-08}$&$9.19\cdot 10^{-08}$&$9.80\cdot 10^{-08}$\\
				&ERR&$1.06\cdot 10^{-06}$&$8.52\cdot 10^{-06}$&$2.37\cdot 10^{-05}$&$2.08\cdot 10^{-04}$\\
				\hline
				 $\bar{\mathcal{P}}$
				&IT&47&31&46&39\\
				&CPU&0.07&0.42&5.88&28.81\\
				&RES&$7.90\cdot 10^{-08}$&$9.79\cdot 10^{-08}$&$9.99\cdot 10^{-08}$&$9.98\cdot 10^{-08}$\\
				&ERR&$1.11\cdot 10^{-05}$&$3.83\cdot 10^{-04}$&$3.14\cdot 10^{-03}$&$2.53\cdot 10^{-03}$\\
				\bottomrule
			\end{tabular}
			\label{tabl4}
		\end{table}	
	
%
\begin{figure}[H]
	\centering
	\includegraphics[width=.7\linewidth]{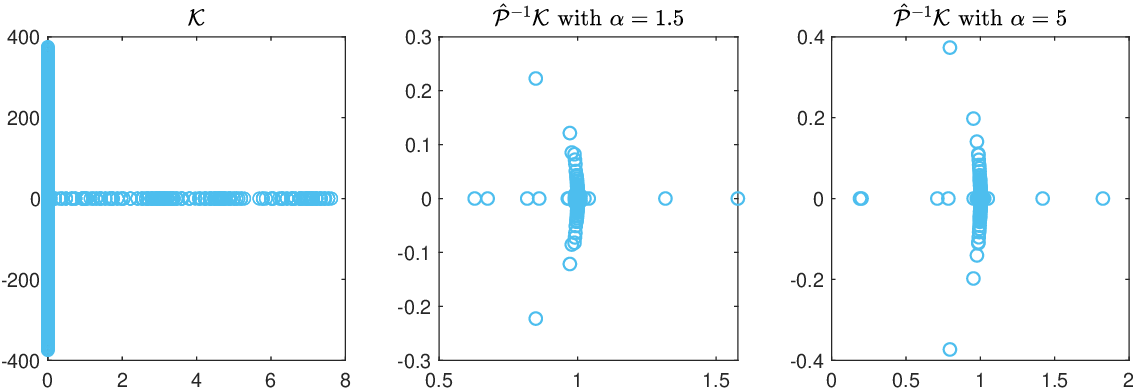}
	\caption{Eigenvalue distribution with $h = \frac{1}{16}$
	 for Example \ref{EX2}.}
	\label{fi2}
\end{figure}
\begin{figure}[H]
	\begin{center}
	\includegraphics[height=5.5cm,width=6.5cm]{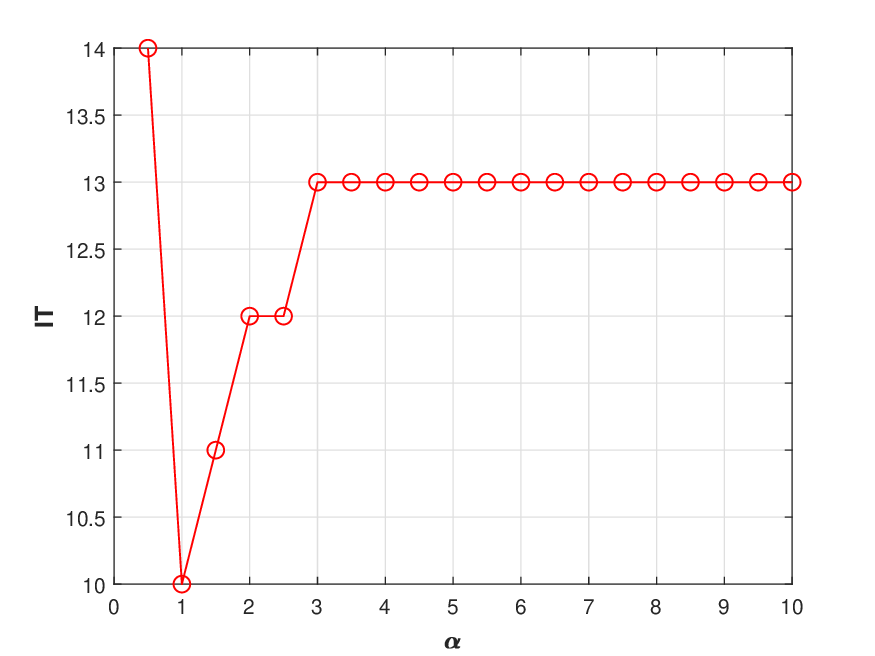}\includegraphics[height=5.5cm,width=6.5cm]{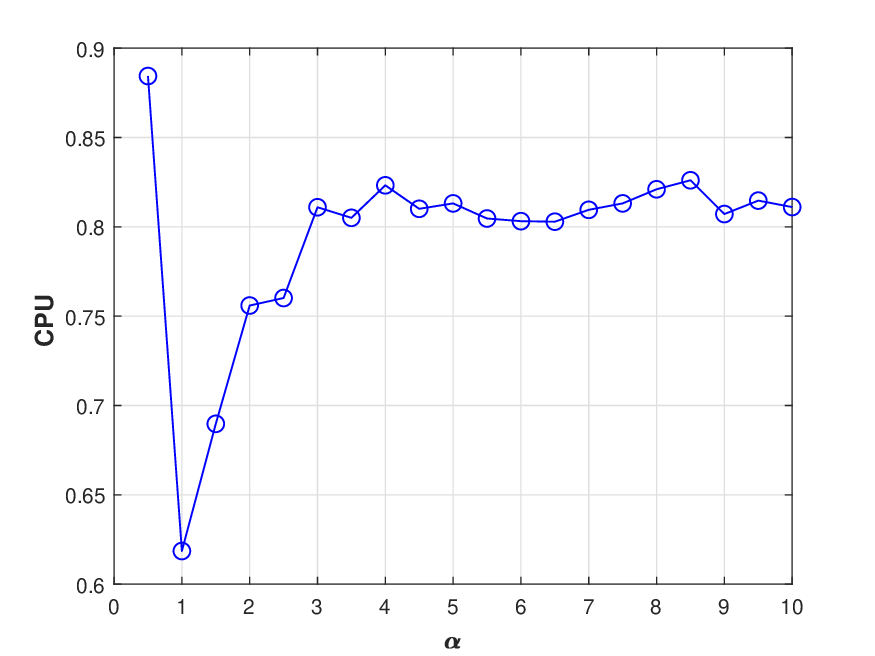}\\
\end{center}
\caption{{\small Variation of  iteration counts (IT) and  CPU time with respect to $\alpha$ for Example  \ref{EX2}}}. \label{EX2FIG} 	
\end{figure}
	\end{example}
	\begin{example}[\cite{Balani,Huang1}]\rm\label{EX3}
		We examine the three-by-three block saddle point problem  \eqref{EQ1} in which
		\[
		A=\mathrm{diag}(2W^{\top}W+D_1,D_2,D_3)\in \mathbb{R}^{n\times n},
		\]
		is a block-diagonal matrix. Additionally,
		$B=[E,-I_{2\tilde{p}},I_{2\tilde{p}}] \in \mathbb{R}^{m\times n} $ and $ C=E^{\top}\in \mathbb{R}^{l\times m}$
		are both full row-rank matrices, where $\tilde{p}=p^2,\hat{p}=p(p+1);W=(w_{i,j})\in \mathbb{R}^{\hat{p}\times \hat{p}}$
		with $w_{i,j}=e^{-2((i/3)^2)+(j/3)^2}; D_1=I_{\hat{P}}$ is an identity matrix; $D_{i}=\mathrm{diag}(d_j^{(i)})\in \mathbb{R}^{2\tilde{p}\times 2\tilde{p}}$, for $ i=2,3$ are diagonal matrices, with
		\begin{align*}
			d_j^{(2)}&=\begin{cases}
				1,&\text{for}\quad 1\leq j\leq \tilde{p},\\
				10^{-5}(j-\tilde{p})^2&\text{for}\quad \tilde{p}+1\leq j\leq 2\tilde{p},
			\end{cases}\\
			d_j^{(3)}&=5\cdot 10^{-11}(j+\tilde{p})^2\,\,\,\text{for}\quad 1\leq j\leq 2\tilde{p},
		\end{align*}
		\[
		E=\begin{pmatrix}
			\hat{E}\otimes I_p\\I_p\otimes \hat{E}
		\end{pmatrix},\quad 
		\hat{E}=\begin{pmatrix}
			2&-1&&&\\
			&2&-1&&\\
			&&\ddots&\ddots\\
			&&&2&-1
		\end{pmatrix}\in\mathbb{R}^{{p}\times (p+1)}.
		\]
		Tables \ref{tabl5} and  \ref{tabl6} present the numerical results for Example \ref{EX3}  using different values of $p$. The results show that the proposed preconditioner for the GMRES method outperforms $\bar{\mathcal{P}}$, requiring fewer iterations and less CPU time. It also performs slightly better than ${\mathcal{Q}}$. Additionally, when using the FGMRES method with the preconditioner  $\hat{\mathcal{P}}$, better CPU times are achieved compared to ${\mathcal{Q}}$.  Furthermore, Figure \ref{fi4} illustrates the eigenvalue distribution of the matrices $\mathcal{K}$ and $\mathcal{\hat{P}}^{-1}\mathcal{K}$ for $p=32$.

		Based on Fig. \ref{EX3FIG}, for $\alpha \geq 0.5$, the iteration count increases steadily and predictably from $12$ to $50$, with the CPU time rising proportionally, reflecting a well‑behaved and reliable convergence behaviour. This monotonic trend shows that the algorithm maintains stable and consistent performance across the entire parameter range, without any abrupt fluctuations or divergence.	The results confirm that the method is robust and flexible, allowing users to choose $\alpha$ based on their desired balance between computational cost and solution accuracy.
		
	\begin{table}[!htp]
	\centering
	\caption{Numerical results of the preconditioned GMRES method  for Example \ref{EX3}.}
	\setlength{\heavyrulewidth}{0.15em} 
	\setlength{\lightrulewidth}{0.08em} 
	\setlength{\cmidrulewidth}{0.08em}
	\begin{tabular}{llllll}
		\toprule
		&Size&8256&32896&131328&524800\\
		&$p$ &32&64&128&256\\
		\hline
		$\hat{\mathcal{P}}(\alpha=1.5)$
		&IT&24&19&16&16\\
		&CPU&0.06&0.10&0.35&1.74\\
		&RES&$6.25\cdot 10^{-08}$&$8.89\cdot 10^{-08}$&$8.21\cdot 10^{-08}$&$9.12\cdot 10^{-08}$\\
		&ERR&$8.29\cdot 10^{-05}$&$5.34\cdot 10^{-05}$&$4.19\cdot 10^{-05}$&$1.23\cdot 10^{-03}$\\
		$\hat{\mathcal{P}}(\alpha=5)$
		&IT&43&38&35&38\\
		&CPU&0.12&0.24&0.83&5.73\\
		&RES&$8.28\cdot 10^{-08}$&$9.26\cdot 10^{-08}$&$9.93\cdot 10^{-08}$&$8.47\cdot 10^{-08}$\\
		&ERR&$1.94\cdot 10^{-04}$&$1.58\cdot 10^{-04}$&$2.35\cdot 10^{-04}$&$6.27\cdot 10^{-04}$\\
		\hline ${\mathcal{Q}}(\alpha=0.5)$
		&IT&44&35&16&6\\
		&CPU&0.22&0.65&1.34&3.27\\
		&RES&$9.53\cdot 10^{-08}$&$9.88\cdot 10^{-08}$&$7.27\cdot 10^{-08}$&$5.63\cdot 10^{-08}$\\
		&ERR&$2.02\cdot 10^{-03}$&$1.82\cdot 10^{-03}$&$1.48\cdot 10^{-03}$&$1.05\cdot 10^{-03}$\\
		${\mathcal{Q}}(\alpha=1.5)$
		&IT&39&34&18&7\\
		&CPU&0.20&0.71&1.40&3.70\\
		&RES&$8.77\cdot 10^{-08}$&$9.15\cdot 10^{-08}$&$7.83\cdot 10^{-08}$&$6.87\cdot 10^{-08}$\\
		&ERR&$1.68\cdot 10^{-03}$&$1.40\cdot 10^{-03}$&$1.36\cdot 10^{-03}$&$1.63\cdot 10^{-03}$\\
		${\mathcal{Q}}(\alpha=5)$
		&IT&42&35&21&10\\
		&CPU&0.19&0.59&1.56&4.70\\
		&RES&$9.08\cdot 10^{-08}$&$9.51\cdot 10^{-08}$&$9.35\cdot 10^{-08}$&$5.38\cdot 10^{-08}$\\
		&ERR&$1.83\cdot 10^{-03}$&$1.60\cdot 10^{-03}$&$1.59\cdot 10^{-03}$&$1.12\cdot 10^{-03}$\\
		\hline $\bar{\mathcal{P}}$
		&IT&125&135&169&257\\
		&CPU&0.59&1.52&9.91&167.68\\
		&RES&$5.74\cdot 10^{-02}$&$4.94\cdot 10^{-05}$&$2.40\cdot 10^{-06}$&$3.35\cdot 10^{-06}$\\
		&ERR&$4.19\cdot 10^{-00}$&$1.01\cdot 10^{-03}$&$9.69\cdot 10^{-04}$&$1.35\cdot 10^{-02}$\\
		\bottomrule
	\end{tabular}
	\label{tabl5}
\end{table}	

\begin{table}[!htp]
	\centering
	\caption{Numerical results of the FGMRES method  for Example \ref{EX3}.}
	\begin{tabular}{llll}
		\toprule
		&Size&524800&2098177\\
		&$p$ &256&512\\
		\hline
		$\hat{\mathcal{P}}(\alpha=1.5)$
		&IT&16&7\\
		&CPU&4.89&8.14\\
		&RES&$8.98\cdot 10^{-08}$&$8.49\cdot 10^{-08}$\\
		&ERR&$1.11\cdot 10^{-03}$&$1.18\cdot 10^{-02}$\\
		
		\hline
		 ${\mathcal{Q}}(\alpha=0.5)$
		&IT&6&4\\
		&CPU&45.69&39.59\\
		&RES&$6.35\cdot 10^{-08}$&$3.14\cdot 10^{-08}$\\
		&ERR&$1.42\cdot 10^{-03}$&$5.37\cdot 10^{-03}$\\
		\hline
		${\mathcal{Q}}(\alpha=1.5)$
		&IT&7&5\\
		&CPU&56.12&88.83\\
		&RES&$7.00\cdot 10^{-08}$&$8.78\cdot 10^{-09}$\\
		&ERR&$1.64\cdot 10^{-03}$&$1.70\cdot 10^{-03}$\\
		\bottomrule
	\end{tabular}
	\label{tabl6}
\end{table}	

\begin{figure}[H]
	\centering
	\includegraphics[width=.7\linewidth]{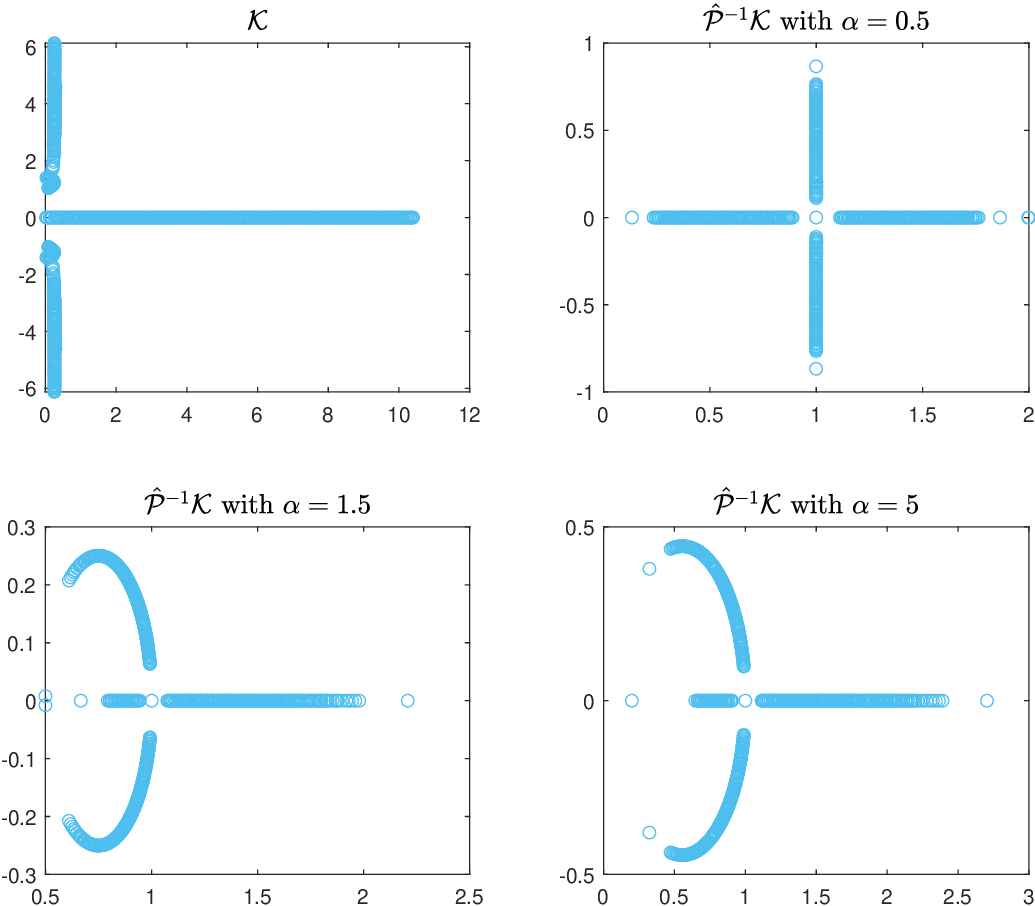}
	\caption{Eigenvalue distribution with $p=32$ for Example \ref{EX3}.}
	\label{fi4}
\end{figure}

\begin{figure}[H]
	\begin{center}
		\includegraphics[height=5.5cm,width=6.5cm]{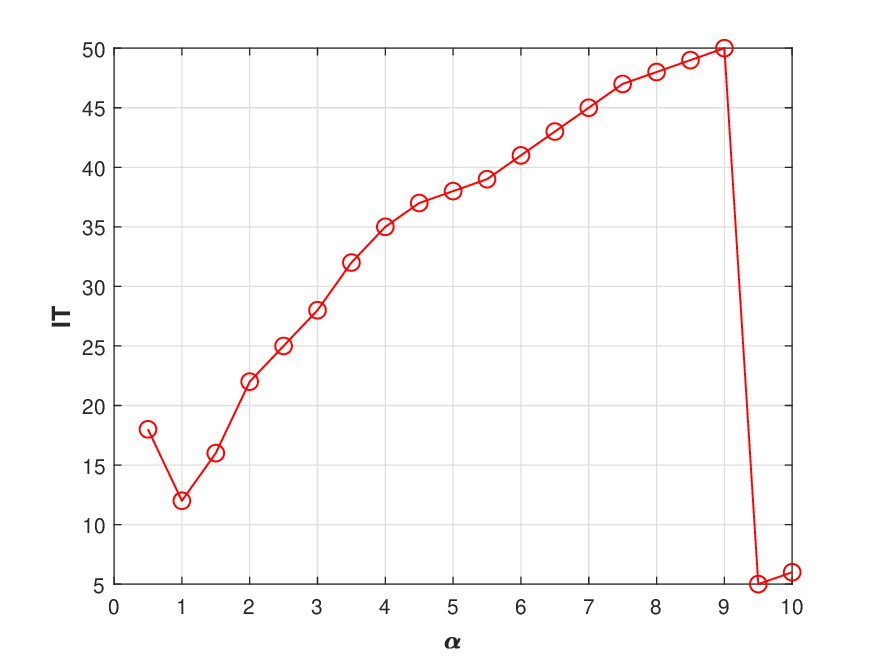}\includegraphics[height=5.5cm,width=6.5cm]{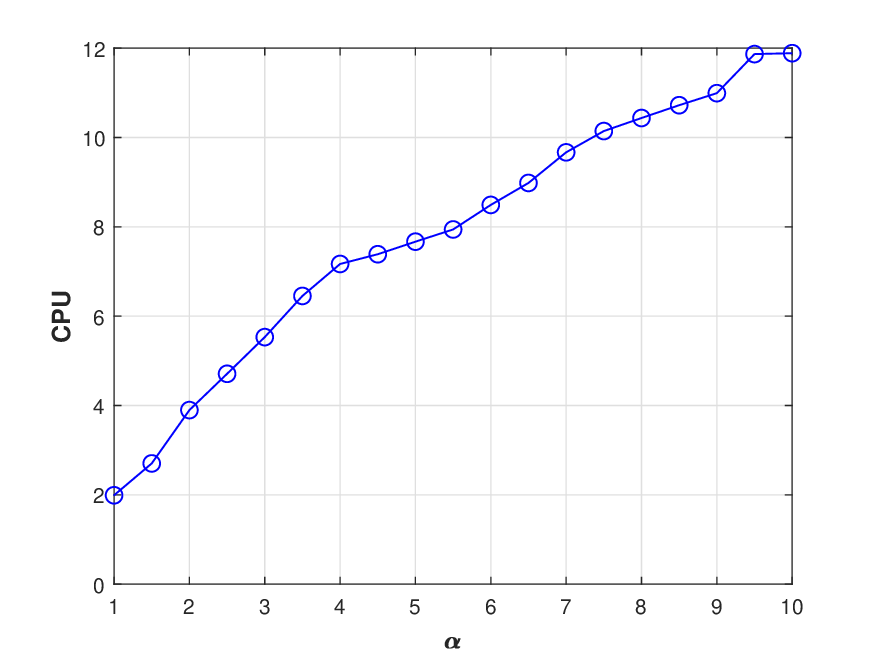}\\
	\end{center}
	\caption{{\small Variation of  iteration counts (IT) and  CPU time with respect to $\alpha$ for Example  \ref{EX3}}}. \label{EX3FIG} 
\end{figure}
	\end{example}	
	
	\section{Conclusions and possible extensions for future work} \label{Sec5}
	
	A biparametric iterative method has been  proposed based on the LU factorization of the coefficient matrix for three-by-three saddle point problems. To ensure convergence, certain provisions have been imposed, with the note that convergence depends on only one of the involved parameters.  By exploiting the induced preconditioner, corresponding eigenpairs of the preconditioned matrix have been derived. The performance of the preconditioner appears to be satisfactory, as indicated by the numerical results.

	The theoretical and numerical results presented in this paper demonstrate the effectiveness of the proposed two-parameter preconditioning framework for a class of three-by-three saddle point problems. Nevertheless, several research directions remain open and deserve further investigation. Some possible directions for future work are listed below:
	
	\begin{itemize}
		\item 
		
		Extending the proposed preconditioner to more general block-structured linear systems with multiple saddle-point constraints and more complex coupling structures.
		
		\item 
		
		Developing an adaptive strategy for selecting the parameters $\alpha$ and $\beta$ based on spectral information, convergence behaviour, or inexpensive estimates of relevant matrix properties. Such a strategy could reduce the need for manual parameter tuning and further improve the efficiency of the proposed preconditioner.

		\item
		
		Investigating parallel implementations of the proposed preconditioning framework for large-scale saddle point problems, with the aim of reducing the computational cost and improving the scalability of the proposed method for large-scale applications.
		
	\end{itemize}
	
		We believe that these directions may further broaden the scope of the proposed preconditioning strategy and provide useful insights for the development of efficient solvers for large-scale saddle point systems.

	\section{Declaration}
	There is no conflict of interest.
	
		\section{Funding}
	The authors declare no funding for this work.
		\section{Data Availability}
The datasets used and/or analysed during the current study are available from the corresponding author on reasonable request.

	\section{Acknowledgment}
	We greatly appreciate  eight anonymous reviewers' insightful feedback, which helped us clarify our arguments and improve the overall presentation.

\end{document}